\documentclass[11pt,a4paper]{amsart}
\usepackage{a4wide}
\usepackage{amsfonts,amsthm,amsmath,amssymb}
\usepackage{amscd,color,caption}
\usepackage{psfrag,graphicx, pgfplots,subfigure,float,subfloat}
\usepackage{import,algorithmic,algorithm2e}
\usepackage{comment}
\usepackage{enumitem}
\usepackage{theoremref}

\usepackage{xcolor}
\usepackage{subfig}
\usepackage{subfigure}
\usepackage{verbatim}

\tikzset{ font={\fontsize{9pt}{12}\selectfont}}
\newcommand{\overbar}[1]{\mkern 1.5mu\overline{\mkern-1.5mu#1\mkern-1.5mu}\mkern 1.5mu}

\newtheorem*{thmA}{Theorem A}
\newtheorem*{thmB}{Theorem B}
\newtheorem*{thmC}{Theorem C}

\theoremstyle{plain}

\newtheorem{theorem}{Theorem}[section]
\newtheorem{lemma}[theorem]{Lemma}
\newtheorem{corollary}[theorem]{Corollary}
\newtheorem{prop}[theorem]{Proposition}

\newtheorem{definition}{Definition}
\newtheorem{remark}{Remark}

\newcommand{\rs}{\hat{\mathbb{C}}}

\newcommand{\bl}{\mathcal{A}^{*}_{\la}(\infty)}
\newcommand{\tl}{T_{\la}}
\newcommand{\ala}{A_{\la}}
\newcommand{\dl}{D_{\la}}
\newcommand{\uc}{\mathcal{U}_{\nu}}
\newcommand{\modulu}{|\la |}
\newcommand{\cl}{\nu_{\lambda}}

\newcommand{\sefu}{\mathcal{C}}

\newcommand{\la}{\lambda}
\newcommand{\La}{\Lambda}
\newcommand{\jl}{\mathcal{S}_{n,d,\lambda}}
\newcommand{\jlz}{\mathcal{S}_{n,d,\lambda_0}}

\newcommand{\jn}{S_{n,a,Q}}

\newsavebox{\savepar}

\begin{document}

\title[Achievable connectivities of Fatou components for a family of rational maps]{Achievable connectivities of Fatou components for a family of singular perturbations}

\date{\today}

\author{Jordi Canela}
\address{Departament de Matem\`atiques, Universitat Jaume I, 12071 Castell\'o de la Plana, Spain}
\email{canela@uji.es}
\author{Xavier Jarque}
\address{Departament de Matem\`atiques i Inform\`atica at Universitat de Barcelona and Barcelona Graduate School of Mathematics, 08007 Barcelona, Catalonia.}
\email{xavier.jarque@ub.edu}
\author{Dan Paraschiv}
\address{Departament de Matem\`atiques i Inform\`atica at Universitat de Barcelona and Barcelona Graduate School of Mathematics, 08007 Barcelona, Catalonia.}
\email{dan.paraschiv@ub.edu}

\thanks{The first author was supported by Spanish Ministry of Economy
	and Competitiveness, through the María de Maeztu Programme for Units of Excellence in
	R\&D (MDM-2014-0445), by BGSMath Banco de Santander Postdoctoral 2017, and by the project UJI-B2019-18 from Universitat Jaume I. The second and third authors were  supported by MINECO-AEI grant
MTM-2017-86795-C3-2-P. The second author was also supported by AGAUR grant 2017 SGR 1374. The third author was also supported by the Spanish government grant FPI PRE2018-086831.}\

\begin{abstract}
In this paper we study the connectivity of Fatou components for maps in a large family of singular perturbations. We prove that, for some parameters inside the family, the dynamical planes for the corresponding maps present Fatou components of arbitrarily large connectivity and we determine precisely these connectivities. In particular, these results extend the ones obtained in \cite{Can1,Can2}. 
\vspace{0.5cm}


\end{abstract}

\maketitle

\section{Introduction}\label{sec:Introduction}

In the recent decades there has been an increasing interest in studying families of rational maps usually called {\it singular perturbations}. Roughly speaking, a family is called a singular perturbation if it is defined by a {\it base} family (called the {\it unperturbed family} and for which we have a deep understanding of the dynamical plane) plus a {\it local} perturbation, that is, a perturbation which has a significant effect on the orbits of points in some part(s) of the dynamical plane, but a very small dynamical relevancy on other regions.

Singular perturbations, no matter the concrete formulas, have some common properties which make their study interesting. On the one hand, the degree of the unperturbed family is smaller than the degree of the perturbed one. Consequently, one should expect richer dynamics for singular perturbations than for the unperturbed maps. On the other hand, most of this new freedom arising from the perturbation may be captive of the dynamical properties of the unperturbed family. The balance between these two scenarios has become very successful in finding new dynamical phenomena. 

The relation between the topology of the dynamically invariant sets (Fatou and Julia set) and the behaviour of the critical orbit(s) is an important issue when studying the dynamical plane of a particular rational map. A paradigmatic example of this is the Dichotomy Theorem for the quadratic family. In this way, singular perturbations are somehow a perfect scenario to observe new phenomena for the invariant sets with respect to the unperturbed maps, for which we usually observe a tame topology. Indeed, the main goal of this paper is to investigate in this direction and to prove that for a certain family of singular perturbations we can construct examples for which, in the same dynamical plane, there are Fatou components with given arbitrarily high connectivities. 

Let $f:\rs \to \rs$ be a rational map acting on the Riemann sphere. Then $f$ partitions the dynamical plane into $\mathcal{J}(f)$, the Julia set, and $\mathcal{F}(f)$, the Fatou set. The Julia set is closed and coincides with the set of points $z\in \hat{\mathbb C}$ where the family of iterates $\{f^{n}|U\}_{n\geq 0}$ is not a normal family for any  neighbourhood $U$ of $z$. Its complement, the Fatou set,  is an open set and its connected components are called Fatou components. By the No Wandering Domains Theorem, Fatou components of rational maps are either preperiodic or periodic (see \cite{Su}). By the Classification Theorem of periodic Fatou components (see \cite{Mi1}, Theorem 16.1), every periodic Fatou component either belongs to the immediate basin of attraction of an attracting or parabolic cycle, or is a simply connected rotation domain (Siegel disk) or is a doubly connected rotation domain (Herman ring). The existence of periodic Fatou components can by studied by analysing the orbits of critical points, i.e.\ points where $f'$ vanishes. Indeed, every immediate basin of attraction contains, at least, a critical point while Siegel disks and Herman rings have critical orbits accumulating on their boundaries.

The connectivity of a domain $D \subset \rs$ is defined as the number of connected components of its boundary. It is known that periodic Fatou components have connectivity $1,$ $2,$ or $\infty$. Indeed, Siegel disks  have connectivity 1, Herman rings have connectivity 2, and immediate basins of attraction have connectivity 1 or $\infty$. Preperiodic Fatou components can have finite connectivity greater than 2. The first such example, with connectivities 3 and 5,  was presented in \cite{Bear}. Moreover, for any given $n\in\mathbb{N}$, there are examples of rational maps with Fatou components of connectivity $n$. These examples can either be obtained by quasiconformal surgery  (see \cite{BKL}) or by giving explicit families of rational maps (see \cite{QJ} and \cite{Ste}). However, the degree of the rational maps obtained in all previous examples grows rapidly with $n$. To our knowledge, the first example of rational map whose  dynamical plane contains Fatou components of arbitrarily large finite connectivities was presented in \cite{Can1} (see also \cite{Can2}) by using singular perturbations. However, in these papers it is not shown which precise connectivities can actually be attained. The goal of this paper is to study the attainable connectivities for a wider family of singular perturbations which includes the ones studied in \cite{Can1, Can2}.
We also want to remark that while this paper was being prepared we knew that, independently,  professor Hiroyuki has obtained another family of rational maps with Fatou components of arbitrarily large connectivity \cite{Hir}.

Singular perturbations of rational maps were introduced by McMullen in \cite{McM1}. He proposed the study of the family
\begin{equation} \label{eq:McMullen}
Q_{n,d,\la}(z)= z^n+\frac{\la}{z^d},
\end{equation}
where $n,d\geq 2$ and $\lambda\in \mathbb C$,  $|\lambda|$ small. Observe that in \eqref{eq:McMullen} the unperturbed map is the {\it simplest} possible: $z^n$. He considered the case $n=2$ and $d=3$ and he  proved that if $|\lambda|$ is small enough then the Julia set is a Cantor sets of quasicircles (the result actually holds for $n$ and $d$ satisfying $1/n+1/d<1$ (compare \cite{DLU})). Later,  Devaney, Look, and Uminsky \cite{DLU} considered \eqref{eq:McMullen} as a $\lambda$-family of rational maps and they extended McMullen's result by proving the Escape Trichotomy. More specifically, they showed that if all critical points belong to the basin of attraction of infinity then the Julia set is a Cantor set, a Sierpinski carpet, or a Cantor set of quasicircles (McMullen's case). The proof relays on the fact that there is a symmetry in the dynamical plane which implies that there is a unique free critical orbit (the symmetry forces all critical points to follow symmetric orbits). Other models similar to \eqref{eq:McMullen} have also been considered. For instance, in \cite{BDGR,GMR} the authors consider singular perturbations of polynomials of the form $z^n+c,\ c\in \mathbb C$, choosing $c$ appropriately. Those examples have shown Julia and Fatou sets with new and rich topology, but the connectivity of the Fatou components is kept as 1, 2 or $\infty$. 

The examples mentioned in the previous paragraph are done perturbing maps with no free critical points: one or more poles are added to superattracting cycles which contain no  critical points, other than the critical points of the cycle, in their basins of attraction. A next natural step is to consider singular perturbations of maps with free critical points. A good candidate for such a perturbation is the family of Blaschke products
$$
B_{n,a}(z)= z^n\frac{z-a}{1-\overline{a}z}, \quad \mbox{where } a\in \mathbb C \mbox{ and } n\geq 2.
$$
See \cite{CFG1} for an introduction to the dynamics of these maps for $n=3$. If $a$ belongs to the punctured unit disk $\mathbb{D}^*:=\mathbb{D}\setminus\{0\}$, the maps $B_{n, a}$ restrict to automorphisms of the unit disk whose dynamical plane is trivial. Indeed, its Fatou set consists of two invariant components given by the immediate basin of attraction $\mathcal{A}^*(0)$ of $z=0$ (the unit disk) and the immediate basin of attraction  $\mathcal{A}^*(\infty)$ of $z=\infty$ (the complement of the closed unit disk). Their common boundary component, the unit circle, is the Julia set of these maps. Moreover, if $a\in \mathbb{D}^* $ the map $B_{n, a}$ has only two simple critical points $c_-\in \mathcal{A}^*(0)$ and $c_+\in \mathcal{A}^*(\infty)$, other than the superatracting fixed points $z=0$ and $z=\infty$. In \cite{Can1,Can2} the author studied the family of singular perturbations of the maps $B_{n,a}$ given by
$$
B_{n,d,a,\la}(z)=z^n\frac{z-a}{1-\overline{a}z}+\frac{\la}{z^d}, 
$$
where $a\in \mathbb{D}^*$ and $\lambda\in\mathbb{C}^*:=\mathbb{C}\setminus\{0\}$, for $n=3$ and $d=2$. Compared to McMullen's singular perturbations, these maps can present a much richer dynamics since their free critical points (which come from the the singular perturbation and the continuous extension of $c_{\pm}$) are not tied by any kind of symmetry. Despite that,  in \cite{Can1} it was proven that if $|\lambda|$ is small enough the family $B_{3,2,a,\la}(z)$ is essentially unicritical: all critical points but the continuous extension $c_-(\lambda)$ of $c_-$ belong to the basin of attraction of infinity $\mathcal{A}(\infty)$. In that case, if $c_-(\lambda)$ belongs to a Fatou component in $\mathcal{A}(\infty)$ which surrounds $z=0$, the dynamical plane has Fatou components of arbitrarily large finite connectivity. The actual existence of parameters for which this actually happens was proven in \cite{Can2}. 
We want to point out that the same results can be proven for $n,d\geq 2$ such that $1/n+1/d<1$. 
In Figure \ref{fig:blaschke} we illustrate the dynamical plane of $B_{n,d,a,\la}(z)$ for $a=0.5$, $d=3$, and different values of $n$ and $\lambda$.

The goal of this paper is to extend the results in \cite{Can1,Can2} to a wider family of singular perturbations and to study which connectivities are attainable for such family. With this aim we consider the family of degree $n+1$ rational maps given by
\begin{equation}
\jn(z)=\frac{z^n(z-a)}{Q(z)}, 
\end{equation} 
where $n\geq 2$, $a\in \mathbb{C}^{*}$, and $Q$ is a polynomial of degree at most $n$. On the one hand it is clear that the family $\jn$ contains the family $B_{n,a}$. On the other hand it is worth to be noticed that $\jn$ also includes the family 
$$
M_{n,a}(z)=z^n(z-a), 
$$
where $n\geq 2$ and  $a\in \mathbb{C}$. This family was first introduced by Milnor in 1991 (see \cite{Mi2}) when studying cubic polynomials ($n=2$) and was later studied by Roesch \cite{Ro1} for $n\geq2$. If $a\neq0$, these maps have $z=0$ and $z=\infty$ as superattracting fixed points of local degree $n$ and $n+1$, respectively. Moreover, they have a unique free critical point $c_a\ne 0$ and the global phase portrait settles down on its dynamical behaviour. It is easy to see that, if $|a|$ is small enough, $c_a$ belongs to the immediate basin of attraction of $z=0$ and the Julia set consists of a quasicircle which separates the immediate basins of attraction of $z=0$ and $z=\infty$ (see \thref{cor:quasicircle}). In this sense, for $|a|$ small the family $M_{n,a}$ can be understood as a simplified version of $B_{n,a}$,  $|a|<1$, where there is no free critical point in $\mathcal{A}^*(\infty)$ but the Julia set is a quasicicle instead of a circle.   

\begin{figure}
	\centering
	
	{
		\begin{tikzpicture}
			\begin{axis}[width=0.49\textwidth, axis equal image, scale only axis,  enlargelimits=false, axis on top]
				\addplot graphics[xmin=-1.1,xmax=1.1,ymin=-1.1,ymax=1.1] {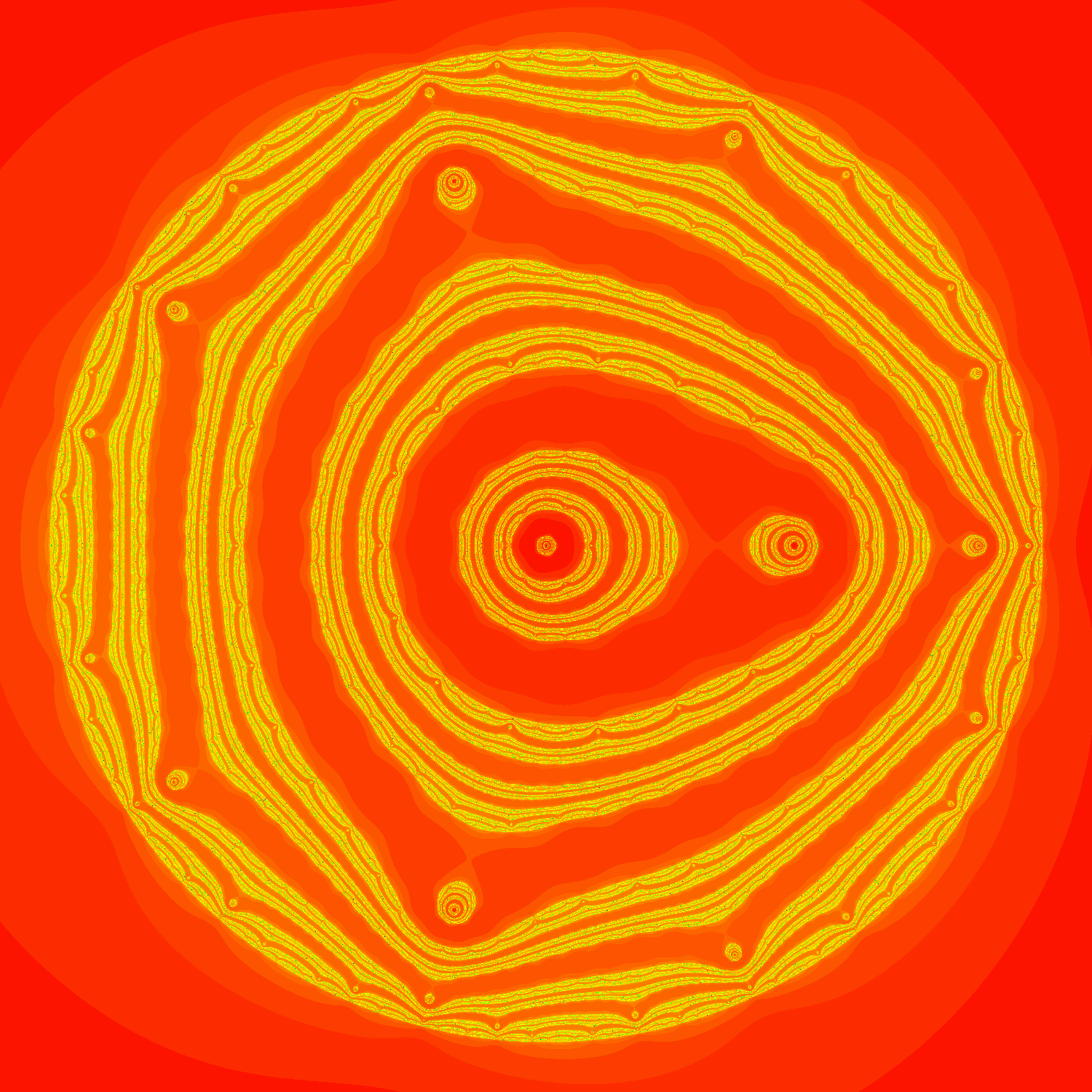};
			\end{axis}
		\end{tikzpicture}
	}
	\hfill
	{
		\begin{tikzpicture}
			\begin{axis}[width=0.49\textwidth, axis equal image, scale only axis,  enlargelimits=false, axis on top]
				\addplot graphics[xmin=-1.1,xmax=1.1,ymin=-1.1,ymax=1.1] {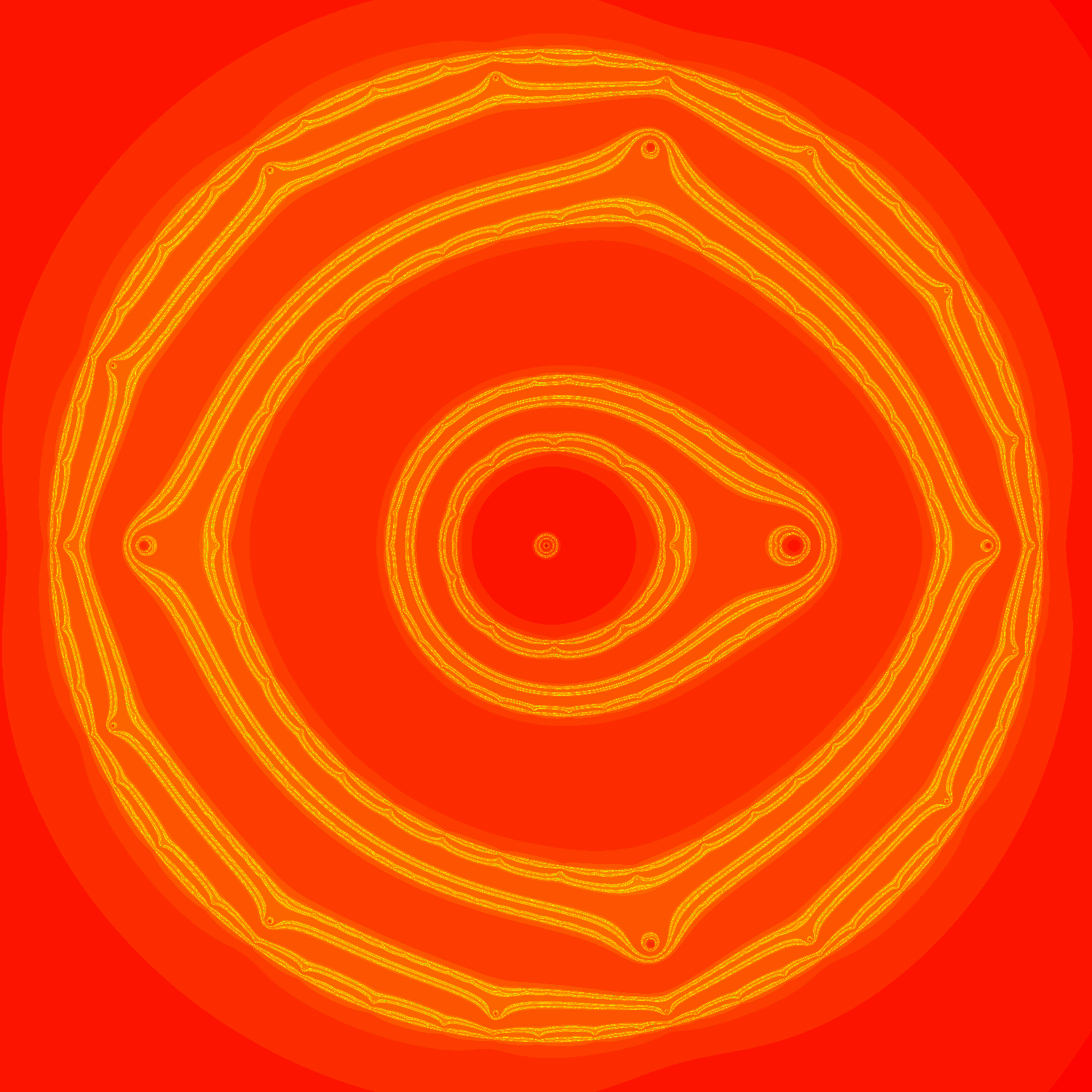};
			\end{axis}
		\end{tikzpicture}
		
	}
	
	\vglue 0.5truecm 
	
	{
		\begin{tikzpicture}
			\begin{axis}[width=0.49\textwidth, axis equal image, scale only axis,  enlargelimits=false, axis on top]
				\addplot graphics[xmin=-1.1,xmax=1.1,ymin=-1.1,ymax=1.1] {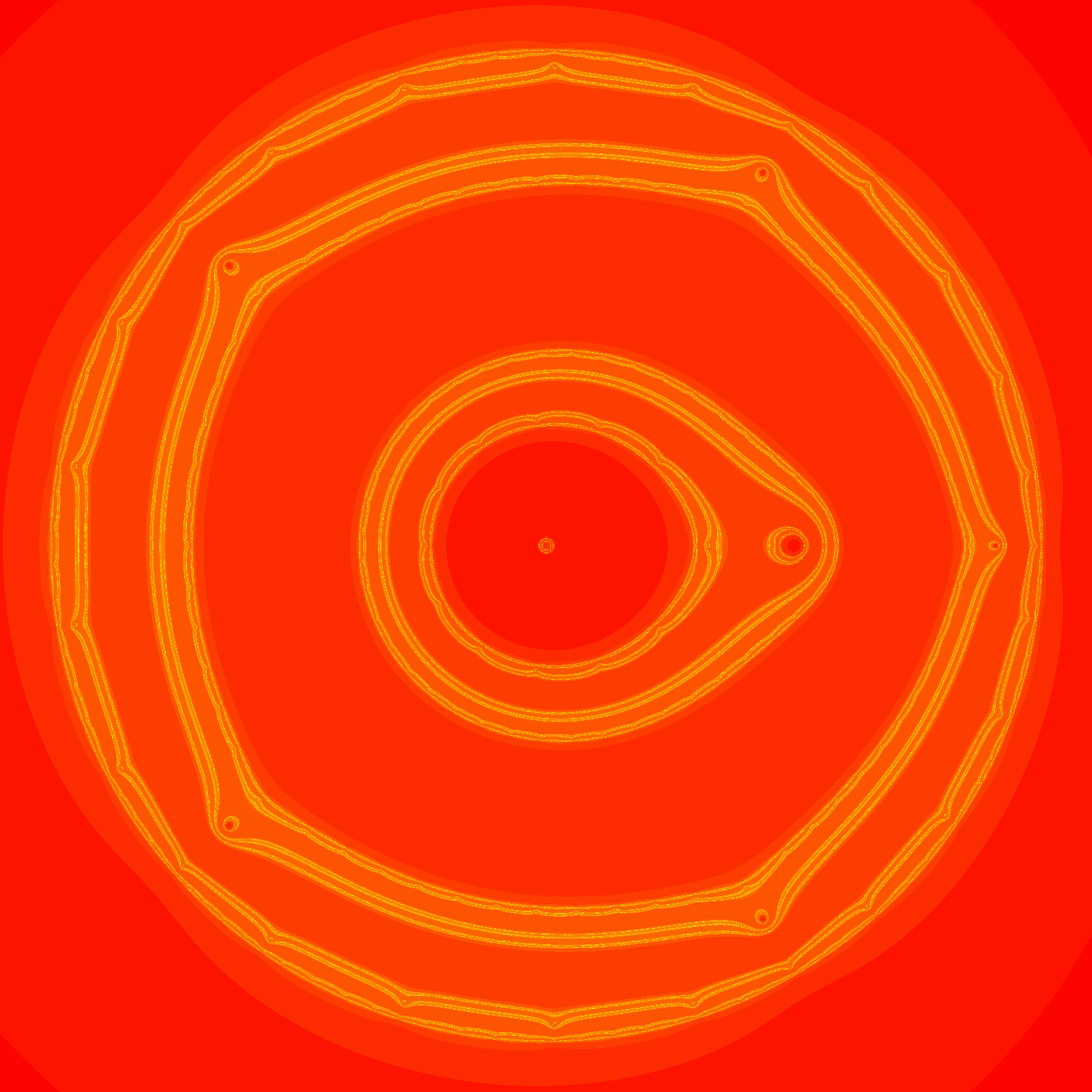};
			\end{axis}
		\end{tikzpicture}
	}
	\hfill
	{
		\begin{tikzpicture}
			\begin{axis}[width=0.49\textwidth, axis equal image, scale only axis,  enlargelimits=false, axis on top]
				\addplot graphics[xmin=-1.1,xmax=1.1,ymin=-1.1,ymax=1.1] {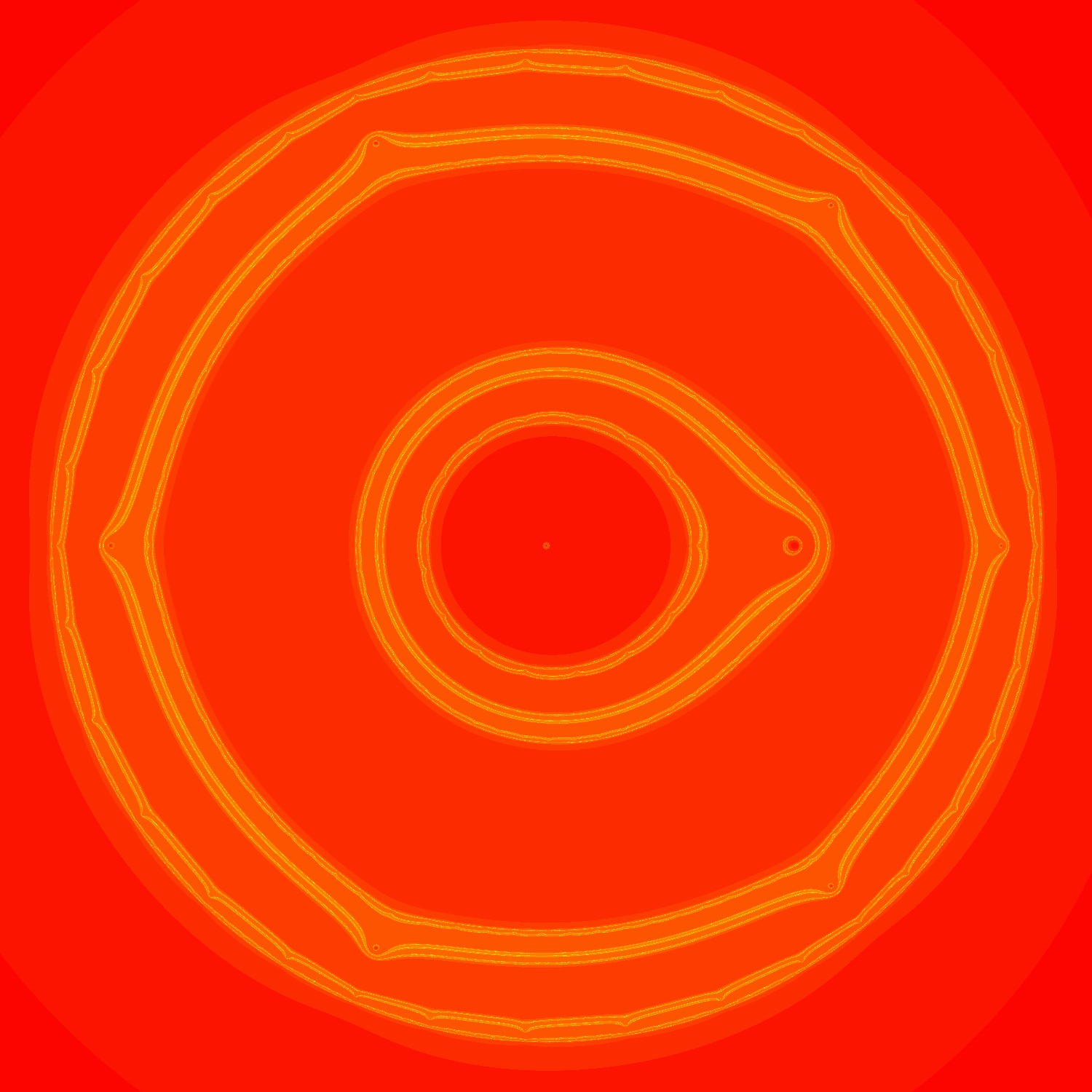};
			\end{axis}
		\end{tikzpicture}
		
	}
	
	\caption{Dynamical planes of the family $B_{n,d,a,\la}(z)$ for $d=3$. The top-left figure corresponds to $n=2$ and $\lambda=2\cdot10^{-8}$; the top-right corresponds to $n=3$ and $\lambda=-5\cdot10^{-8}$; the bottom-left figure corresponds to $n=4$ and $\lambda=-6.3\cdot10^{-9}$; and bottom-right corresponds to $n=5$ and $\lambda=-1.2\cdot10^{-10}$. In all cases we can see the triply connected regions (where the critical point $\nu_\lambda$ lies) and  their eventual preimages, which are Fatou components with increasing connectivity.}
	\label{fig:blaschke} 
\end{figure}

We now turn to the unperturbed family $\jn$. Inspired by the work in \cite{Can1,Can2} we will impose the following conditions to be satisfied for the maps in $\jn$.

\begin{enumerate}[label=(\alph*)]

\item{The point $z=0$ is a superattracting fixed point of degree $n$ of $\jn$. In particular $Q(0) \ne 0$.}
\item{The fixed point $z=\infty$ is (super)attracting. In particular the coefficient of $z^n$ of $Q$, say $b_n$, satisfies $0\leq |b_n|<1$.}
\item{There are exactly two Fatou components: the immediate basins of attraction $\mathcal{A}^*(0)$ and $\mathcal{A}^*(\infty)$ of $z=0$ and $z=\infty$, respectively.}

\end{enumerate}

\begin{remark}
We can deduce the following observations from the above conditions. Since the maps $\jn$ have degree $n+1$, the immediate basins of attraction are mapped onto themselves with degree $n+1$ and, hence, each of them contains exactly $n$ critical points counting multiplicity. In particular, the basin of attraction of $z=0$ (which is a critical point of multiplicity $n-1$) contains a simple critical point $\nu_0\ne 0$. 
\end{remark}

Once the unperturbed family has been described, we now consider the singular perturbation 
\begin{equation}
\jl(z)=\jn(z)+ \frac{\lambda}{z^d}, \quad \la \in \mathbb{C}^{*}, \ d\geq 2.
\end{equation} 
Notice that to simplify notation we do not specify the dependence on $a$ and $Q$ of the family $\jl$. Notice also that the family $\jl$ includes the family $B_{n,d,a,\la}$.  It follows immediately that all maps  $\jl$ have degree $n+d+1$ and that for $\lambda \ne 0$ the point $z=0$ is a pole of degree $d$. 
We will say that {\it $\jl$ satisfies (a), (b), and (c)} if $\jn$ satisfies the conditions (a), (b), and (c) explained above. Analogously to the condition needed to obtain a Cantor set of quasicircles for McMullen's family (see \cite{DLU}), we have to add a fourth condition to the family: 

\begin{enumerate}
\item[(d)]{The numbers $n,d\geq2$ are such that $\frac{1}{n}+\frac{1}{d}<1$. In other words, we exclude $n=d=2$.}
\end{enumerate}

 Since the critical points are not tied by any relation, for $|\lambda|$ big the dynamics can be very rich. Despite that, if $|\lambda|$ is small the family is essentially unicritical. Indeed, there exists a constant $\mathcal{C}>0$ such that if $|\lambda|<\mathcal{C}$, $\lambda\neq 0$,  the following hold (see \thref{munnum}):
\begin{itemize}
	\item The continuous extensions of the $n$ critical points which belong to the immediate basin of attraction $\mathcal{A}_0^*(\infty)$ of $z=\infty$ before perturbation belong to the immediate basin of attraction $\mathcal{A}_{\lambda}^*(\infty)$ of $z=\infty$ after perturbation. Moreover, $\mathcal{A}_{\lambda}^*(\infty)$ is a quasidisk.
	\item The pole $z=0$ belongs to a quasidisk $T_{\lambda}$ (usually called {\it trap door}) which is mapped onto $\mathcal{A}_{\lambda}^*(\infty)$ under $\jl$.
	\item The $n+d$ critical points which appear around $z=0$ after perturbation belong to a doubly connected Fatou component $A_{\lambda}$ which is mapped onto $T_{\la}$ under $\jl$.
\end{itemize}

The previous points actually coincide with the skeleton of the dynamics in the Cantor set of quasicircles case of McMullen's family (see \cite{McM1}). This is why the dynamical planes for this perturbed family resemble the dynamical planes for the Cantor set of quasicircles with extra decorations (see Figure \ref{fig:blaschke} and Figure \ref{fig:milnor}). These decorations come from the presence of the extra critical point $\nu_{\la}$, which comes from the continuous extension of the critical point $\nu_0$ that belongs to the basin of attraction of $z=0$ before perturbation. This is the only critical point which may not belong to the basin of attraction $\mathcal{A}_{\lambda}(\infty)$ of $z=\infty$ after perturbation if $|\lambda|$ is small. We want to remark that  the main difference between $\jl$ and the particular family $B_{n,d,a,\la}$ is that we allow certain degrees of freedom in the $n$ critical points that lie in  $\mathcal{A}^*_{\lambda}(\infty)$. For instance, if the degree of $Q$ is 0, then $z=\infty$ is a superatracting fixed point of local degree $n$. On the other hand,  if the degree of $Q$ is $n$, then $z=\infty$ is attracting (but not superattracting) and there are $n$ critical points which move in $\mathcal{A}^*_{\lambda}(\infty)$. Also, the shape of the Julia set before perturbation affects the shape of the Julia set of the perturbed map (see Figure \ref{fig:blaschke} and Figure \ref{fig:milnor}). Recall that in the Blaschke case the unperturbed Julia set is the unit circle.

\begin{figure}
	\centering
	{
		\begin{tikzpicture}
			\begin{axis}[width=0.49\textwidth, axis equal image, scale only axis,  enlargelimits=false, axis on top]
				\addplot graphics[xmin=-1,xmax=1.6,ymin=-1.1,ymax=1.5] {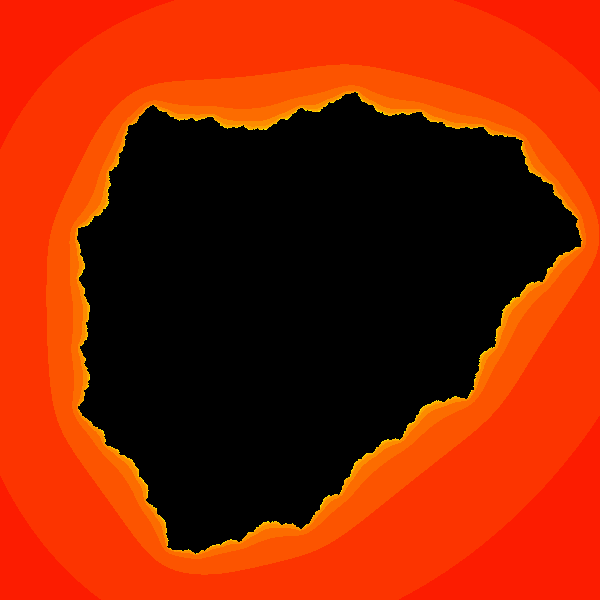};
			\end{axis}
	\end{tikzpicture}}
	\hfill
	{\begin{tikzpicture}
			\begin{axis}[width=0.49\textwidth, axis equal image, scale only axis,  enlargelimits=false, axis on top]
				\addplot graphics[xmin=-1,xmax=1.6,ymin=-1.1,ymax=1.5] {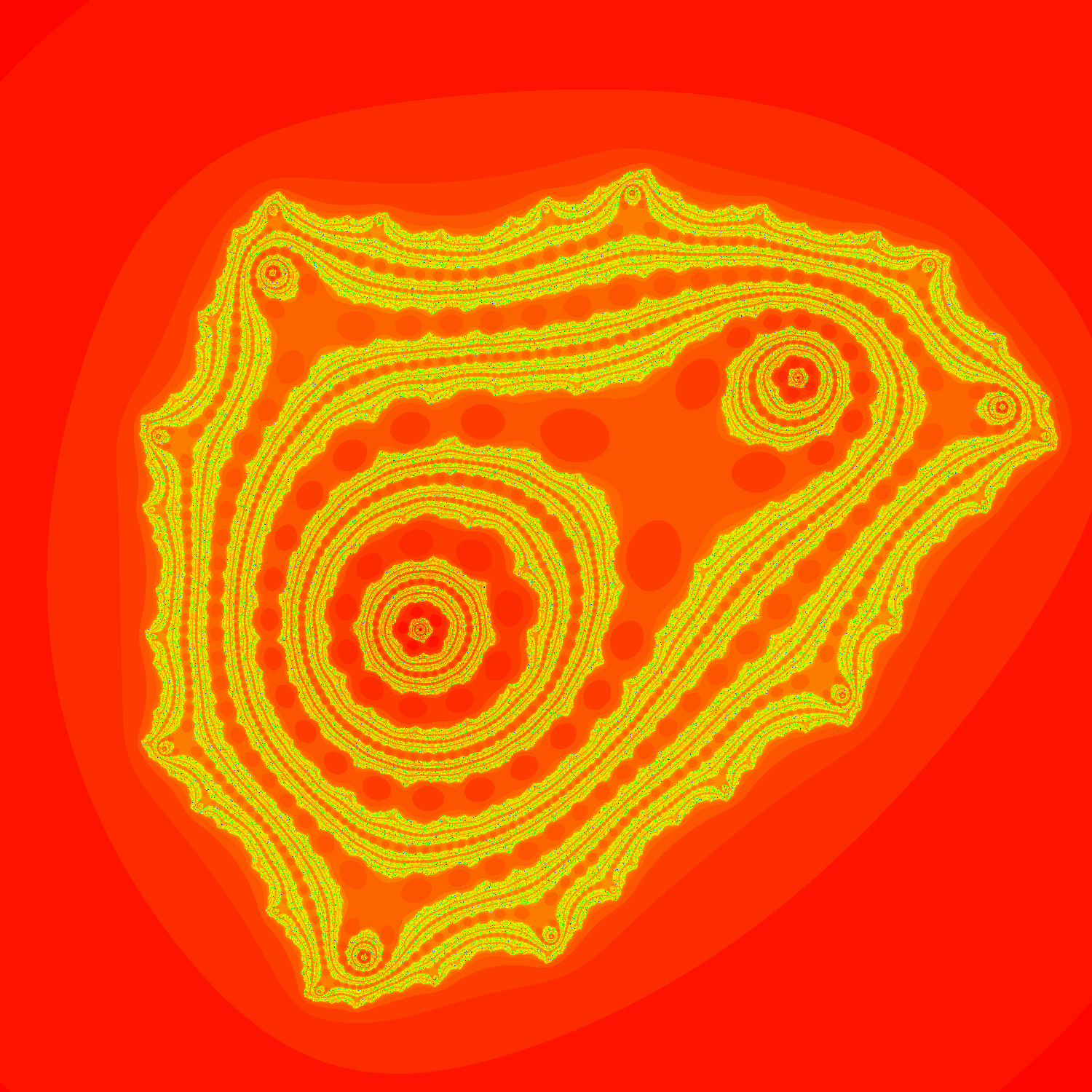};
			\end{axis}
	\end{tikzpicture}}
	\caption{Left figure illustrates the dynamical planes of $M_{n,a}$ for $n=2$ and $a=(0.9+0.6i)$. Right picture illustrates the dynamical plane of the (perturbed) family $\jl$ when the unperturbed map is precisely $M_{2,a}$, and the pertubation corresponds to $d=3$ and $\la=-10^{-7}$. We can see in the right figure the triply connected Fatou component which contains $\nu_{\la}$ and its  eventual preimages with higher connectivity.}
	\label{fig:milnor}
\end{figure}

The goal of the paper is to analyse the connectivities which can be achieved with these singular perturbations. The critical point $\nu_{\la}$ is crucial in order  to increase the connectivities beyond 2. 
Indeed, if $\nu_{\lambda}$ belongs to a preimage $U_{\nu}$ of $A_{\lambda}$ then the Fatou component $U_{\nu}$ is triply connected. 
Moreover, if $U_{\nu}$ surrounds $z=0$ then we can find sequences of iterated preimages of $U_{\nu}$ which increase the connectivity with every iteration. 
The next theorems describe the connectivities which can be achieved with this process. 
We denote by  $\textrm{Bdd}(\ala)$ the union of the connected component of the complement of $\ala$ not containing $z=\infty$ and the annulus itself.

\begin{thmA}
Let $\jl$ satisfying (a), (b), (c), and (d) and let $\lambda\neq 0$, $\modulu<\sefu$. Assume also that $\cl \in \uc$, where $\uc$ is an iterated preimage of $\ala$ which surrounds $z=0$. Let $k$ be the minimal number of iterations needed by the free critical point $\cl$ to be mapped into $\textrm{Bdd}(\ala)$. Let $U$ be a Fatou component of connectivity $\kappa>2$. Then, there exist $i,j,\ell \in \mathbb{N}$ such that $\kappa=(n+1)^{i}d^j n^\ell+2$ and $\ell \le jk$.
\end{thmA}

In other words, Theorem A is telling us all {\it potential} connectivities 
$\kappa>2$ for a Fatou component of a map in $\jl$  for  
$|\lambda|$ sufficiently small; but it is not claiming the existence of a Fatou component of each $(i,j,\ell)$-connectivity. The next result complements Theorem A and it gives the connectivities that are certainly achieved for any parameter $\lambda$ as long as $|\lambda|$ is sufficiently small and $\nu_{\la}$ satisfies certain dynamical conditions. 

\begin{thmB}
Let $\jl$ satisfying (a), (b), (c), and (d) and let $\lambda\neq 0$, $\modulu<\sefu$. Assume also that $\cl \in \uc$, where $\uc$ is an iterated preimage of $\ala$ which surrounds $z=0$. Let $k\ge 1$ be the minimal number of iterations needed by the free critical point $\cl$ to be mapped into $\textrm{Bdd}(\ala)$. For any given $i,j,\ell \in \mathbb{N}$ such that $\ell \le j(k-1)$, there exists a Fatou component $U$ of connectivity $\kappa=(n+1)^{i}d^j n^\ell+2$.
\end{thmB}

In Theorem A  and Theorem B the achievable connectivities depend on the minimal number of iterations $k>0$ needed by the free critical point $\cl$ to be mapped into $\textrm{Bdd}(\ala)$. However, choosing $\lambda$ appropriately we can make this $k$ as big as desired. Therefore, for any $\ell$ and $j$ we can find $\lambda$ so that the inequality $\ell \le j(k-1)$ is satisfied. From this, we obtain Theorem C.

\begin{thmC}
Let $\jl$ satisfying (a), (b), (c), and (d) and let $\lambda\neq 0$, $\modulu<\sefu$. For any given $i,l \ge 0$ and $j>0$, there exists a parameter $\la$ such that $\jl(z)$ has a Fatou component of connectivity $\kappa=(n+1)^{i}d^j n^\ell+2$, and a Fatou component of connectivity $\kappa=(n+1)^i+2$.
\end{thmC}

The paper is organised as follows. In Section~\ref{sec:preliminaries} we briefly introduce the tools later used in the paper. In Section \ref{sec:perturbed} we describe in detail the skeleton of the dynamical plane of $\jl$ satisfying the conditions (a), (b), (c) and (d) for $|\la|$ small enough. In Section \ref{sec:Theorem_AB} we prove Theorems A and B. Finally, in Section \ref{sec:Theorem_C} we prove Theorem C.

\section{Preliminaries}\label{sec:preliminaries}

In this section we present the main tools that we use along the paper. Before, we introduce some notation. In the introduction we used the notation $\textrm{Bdd}(A)$ to denote the set bounded by an annulus $A$, including itself. It will be useful to generalise this concept for other multiply connected sets. Let $U\subset \mathbb{C}$ be a multiply connected open set. We denote by $\textrm{Bdd}(U)$ the minimal simply connected open set which contains $U$ but not $z=\infty$. Let $\gamma\in\mathbb C$ be a Jordan curve. We denote by ${\rm Ext}\left(\gamma\right)$ and ${\rm Int}\left(\gamma\right)$ the connected components of $\hat{\mathbb C} \setminus \gamma$ that contain $z=\infty$ and do not contain $z=\infty$, respectively.
We denote by $A ( \gamma_1, \, \gamma_2)$ the open annulus bounded by Jordan curves $\gamma_1$ and $\gamma_2$ with $\gamma_1 \subset \textrm {Int}\left(\gamma_2\right)$. We denote the circle centered at the origin and of radius $c>0$ by $\mathbb{S}_c$. Finally, if $U \subset \mathbb C$ we denote by $\overline{U}$ its closure.

We now proceed to introduce the needed tools. The next result provides a sufficient criterion for the Julia set of a map  to be a quasicircle.

\begin{theorem}{ \cite[Theorem 2.1, page 102]{CaGa} }
\thlabel{CarGam}If the Fatou set of a rational map $R$ contains exactly two Fatou components and the map $R$ is hyperbolic on its corresponding Julia set $J(R)$, then $J(R)$ is a quasicircle.

\end{theorem}

We can immediately conclude that the Julia sets of the maps $\jn$ are quasicircles.

\begin{corollary}\thlabel{cor:quasicircle}
\thlabel{2comp}Let $\jn$ satisfying (a), (b), and (c). Then, its Julia set is a quasicircle.

\end{corollary}

Finally we recall the Riemann-Hurwitz formula (see for instance \cite{Ste}), which we use in order to study connectivities of Fatou components. 

\begin{theorem}{(Riemann-Hurwitz formula)}
\thlabel{RH}Let $U,V \subset \hat{\mathbb{C}}$ be two connected domains of connectivity $m_U, m_V \in \mathbb{N}^{*}$ and let $f: U \to V$ be a degree $d$ proper map branched over $r$ critical points, counted with multiplicity. Then, 
$$m_U-2=d(m_V-2)+r .$$

\end{theorem}

Along the text we also use the following corollary of the Riemann-Hurwitz formula (compare \cite[Corollary 2.2]{CFG1}).
\begin{corollary}
\thlabel{corocoro}
Let $U\subset \hat{\mathbb U}$ be an open set and let $f: U \to f(U)$ be a proper holomorphic map. Then, the following statements hold:
\begin{enumerate}[label=(\roman*)]

\item{If $f(U)$ is doubly connected and $f$ has no critical points in $U$, then $U$ is doubly connected.}
\item{If $f(U)$ is simply connected and $f$ has at most one critical point in $U$ (not counting multiplicities), then $U$ is simply connected. }

\end{enumerate}
\end{corollary}

\section{The perturbed family}\label{sec:perturbed}

Let $\jl$ satisfying conditions (a), (b), and (c) described above. This section describes the main properties of the dynamical plane for parameters belonging to a neighbourhood of $\la=0$. We first describe the immediate basin of attraction of $z=\infty$, which we further denote by $\bl$, and its boundary. The proof of \thref{prop:continfty} uses the theory of polynomial-like mappings introduced by Douady and Hubbard in \cite{DH1} (see also \cite{BF}).
 
\begin{remark}
Along the paper, when we say that a compact set moves continuously with respect to parameters, we use the topology induced by the Hausdorff metric for compact sets.
\end{remark}

\begin{prop}
\thlabel{prop:continfty}
Let  $\jl$ satisfying conditions (a), (b), and (c). Then, for $\modulu$ small enough, the following hold:
\begin{enumerate}[label=(\roman*)]

\item{The Fatou component $\bl$ is mapped onto itself with degree $n+1$.}
\item{The boundary of $\bl$ is a quasicircle that moves continuously with respect to $\la$.}
\item{The set $\bl$ contains exactly $n$ critical points counting multiplicity. Each critical point of $\jl$ in $\bl$ is a continuous extension of a critical point of $\jn$ in $\mathcal{A}^{*}(\infty)$.}
\end{enumerate}
\end{prop}

\begin{proof}
Observe that the three statements are trivially satisfied (by definition and Corollary \ref{2comp}) for the unperturbed family. So this proposition says that this conditions are still true if the perturbation is small enough. To prove the proposition we  show the existence of an analytic family of polynomial-like maps which ensures the continuous deformation of the key dynamical objects.

Fix $\jn$ and let $U$ be the maximal domain of Bottcher coordinates around $z=0$. The critical point $\nu_0$ lies on $\partial U \subset \mathcal{A}^{*}(0)$. 
Let $\gamma$ be an analytic Jordan curve surrounding the origin such that $\gamma \in U \setminus \overline{\jn(U)}$. We now show that the preimage of $\gamma$ has a unique connected component. Let $A:=A(\gamma ,\, \partial \mathcal{A}^{*}( 0 ))$. Notice that $A\subset \mathcal{A}^{*}(0)$. Observe that the annulus $A$ does not include any critical value. By \thref{corocoro}(i), its preimage $A^{-1}$  (under $\jn$) is also an annulus. Since $\partial \mathcal{A}^{*}( 0 )$ is mapped by $\jn$ onto itself with degree $n+1$, $A^{-1}$ is also mapped onto $A$ with degree $n+1$. Let $\gamma_0^{-1}$ be the connected component of $ \partial A^{-1}$ other than $\partial \mathcal{A}^{*}( 0 )$. Since $A^{-1}$ is mapped onto $A$ with degree $n+1$ under $\jn$, then $\gamma_0^{-1}$ is mapped onto $\gamma$ with degree $n+1$ under $\jn$. Since $\jn$ has (global) degree $n+1$, we conclude that there is no other preimage of $\gamma$ than $\gamma_0^{-1}$ under the map $\jn$.

Let $\mathcal{V}=\textrm{Ext}(\gamma)$ and $\mathcal{U}_{0}=\textrm{Ext}(\gamma_{0}^{-1})$. It follows from the construction that $\mathcal{U}_{0}$ is compactly contained in $\mathcal{V}$ and $\jn|_{\mathcal{U}_{0}}:\mathcal{U}_{0}\mapsto \mathcal{V}$ is a proper map of degree $n+1$. Therefore, the triple $(\jn|_{\mathcal{U}_{0}},\mathcal{U}_{0}, \mathcal{V})$ is a degree $n+1$ polynomial-like map (see \cite{DH1}, see also \cite{BF}). We want to extend (for $|\lambda|$ small enough) this map to  a $\mathcal{J}$-stable analytic family of polynomial like mappings. Observe that the map $\jl$ depends analytically on $\lambda$ for all $z\in\widehat{\mathbb{C}}\setminus\mathbb{D}_{\varepsilon}$, where $\mathbb{D}_{\varepsilon}$ denotes the disk of radius $\varepsilon$ centered at $z=0$, for all $\lambda\in\mathbb{C}$ and all $\varepsilon>0$. Recall that $\jn=\mathcal{S}_{n,d,0}$. Therefore, if $|\lambda|$ is small enough, the continuous extensions of the $n$ critical points (counting multiplicity) which lie in $\mathcal{A}^{*}(\infty )=\mathcal{A}_0^{*}( \infty )$ for $\jn$ lie in  $\mathcal{A}_{\lambda}^{*}( \infty )$ for $\jl$.  Moreover, if $|\lambda|$ is small enough then there exists a unique connected component $\gamma_{\la}^{-1}$ of $\gamma$ under the map $\jl^{-1}$ which is an analytic Jordan curve. In fact, $\gamma_\lambda^{-1}$ is a continuous deformation of $\gamma_{0}^{-1}$ and it is mapped with degree $n+1$ onto $\gamma$ under $\jl$. Let $\mathcal{U}_{\la}=\textrm{Ext}(\gamma_{\la}^{-1})$. Decreasing $|\lambda|$ if necessary, we can ensure the following. The set $\mathcal{U}_{\la}$ is compactly contained in $\mathcal{V}$ and the only critical points of $\jl$ in $\mathcal{U}_{\la}$ are the ones which come from the continuous extension of the critical points in $\mathcal{A}^{*}(\infty )$. Moreover, $\jl|_{\mathcal{U}_{\la}}:\mathcal{U}_{\la}\mapsto \mathcal{V}$ is a proper map of degree $n+1$ and the triple $(\jl|_{\mathcal{U}_{\la}},\mathcal{U}_{\la}, \mathcal{V})$ is a degree $n+1$ polynomial-like mapping.

Let $\Lambda$ be an open disk centered  at $\la=0$ compactly contained in the open set of parameters for which the previous conditions hold. Then, $\{(\jl|_{\mathcal{U}_{\la}},\mathcal{U}_{\la}, \mathcal{V})\}_{\la\in\Lambda}$ defines an analytic family of polynomial like mappings (see \cite{DH1}, see also \cite{BF}). Let $\mathcal{K}_{\lambda}:=\{z\in\mathcal{U}_{\la}\;|\; \jl^n(z)\in\mathcal{U}_{\la} \mbox{ for all } n\geq0\}$ and $\mathcal{J}_\la=\partial \mathcal{K}_\la$ denote the filled Julia set and the Julia set of the polynomial like map $(\jl,\mathcal{U}_{\la}, \mathcal{V})$, respectively. Notice that $\mathcal{K}_{0}= \overline{\mathcal{A}^*(\infty)}$. Notice also that all connected components of the interior of $\mathcal{K}_\la$ are Fatou components of $\jl$. Therefore, since the point $z=\infty$ belongs to $\mathcal{K}_\la$ for all $\la\in\Lambda$, we conclude that $\mathcal{A}_{\la}^*(\infty)\subset \mathcal{K}_\la$. 

To finish the proof, we observe that since all critical points of $\jl|_{\mathcal{U}_{\la}}$ belong to $\mathcal{A}_{\la}^*(\infty)$ it follows that the analytic family of $\{(\jl,\mathcal{U}_{\la}, \mathcal{V})\}_{\la\in\Lambda}$ is $\mathcal{J}$-stable. In particular, the Julia sets $\mathcal{J}_\la$ are quasicircles which are continuous deformations of $\mathcal{J}_0=\partial \mathcal{A}^*(\infty)$ (see \cite[Proposition 10]{DH1}). Notice that, by  \thref{2comp}, $\partial \mathcal{A}^*(\infty)$ is a quasicircle. Since $\mathcal{A}_\la^*(\infty)\subset \mathcal{K}_\la$, we can conclude that $\partial \mathcal{A}_\la^*(\infty)=\mathcal{J}_\la$ for all $\la\in\Lambda$. This proves (ii). Statements (i) and (iii) follow from the choice of the set of parameters $\Lambda$.

\end{proof}

The first part of the following lemma describes a neighbourhood of $z=\infty$ which, for $\modulu$ small enough, always lies in the interior of $\bl$. The second part shows that $z=0$ lies in a preimage of $\bl$, different from it.

\begin{lemma}
\thlabel{Klemma}

Let  $\jl$ satisfying conditions (a), (b), and (c). Then, for $\modulu$ small enough, the following happen:
\begin{enumerate}[label=(\roman*)]
\item{There exists a constant $K$, which only depends on $n,$ $a,$ and $Q$, such that $z\in\bl$ if $|z|>K$.}
\item{Assume that $\jl$ also satisfies condition (d). For any constant $K_1>0$, if $\modulu$ is small enough, the disk $\left \{ |z|<K_1 |\la|^{\frac{n+d}{nd} } \right \}$ belongs to a Fatou component $T_{\la}$. The Fatou component $\tl$ is mapped onto $\bl$ and it is different from $\bl$.}
\end{enumerate}
\end{lemma}

\begin{proof}
We begin with statement \textit{(i)}. From condition (b) we know that, for fixed $\jn$, there exists a constant $K$ such that the set $\{ z \in \mathbb{C}| \, |z|>K\}$ is compactly contained in the immediate basin of attraction of $\infty$. By continuity with respect to $\la$, for $\modulu$ small enough, this set is also contained in $\bl$.

For statement \textit{(ii)}, let $K_1>0$. Assume that $\la$ is such that \textit{(i)} is satisfied for the constant $K$ above. Let $z\in \mathbb C$ such that $|z|<K_1 |\la|^{\frac{n}{n+d}}$. It follows that $$|\jl(z)|>\left|\frac{\lambda}{z^d}\right|-\left|\frac{z^n(z-a)}{Q(z)}\right|>\frac{|\la|^{1-\frac{nd}{n+d}}}{K_1^d}- \frac{|\la|^{\frac{1}{n+d}}K_1^n(|a|+1)}{M}=:C_1 (\la)+C_2(\la).$$  Notice that $C_2(\la)$ tends to $0$ as $\la$ tends to $0$. Because of assumption (d), $C_1(\la)$ tends to $\infty$ as $\la$ tends to $0$. Shrinking $\modulu$ if necessary, if $|z|<|\la|^{\frac{n}{n+d}} K_1$, then $|\jl(z)|>K$. We conclude that the set $\left \{ |z|<K_1 |\la|^{\frac{n+d}{nd} } \right \}$ belongs to a Fatou component. This Fatou component contains $z=0$, which is mapped to $\infty$ with degree $d$. By continuity with respect to $\la$ and \thref{prop:continfty}, $\partial \bl$ is a quasicircle which surrounds $z=0$. It follows that $\bl$ does not contain the origin and $z=0$ belongs to a preimage of $\bl$, different from $\bl$, which we denote by $\tl$.

\end{proof}

Recall that each map of the perturbed family has global degree $n+d+1$. Hence, it has $2(n+d)$ critical points (counting multiplicity). By \thref{prop:continfty}, $n+d-1$ of them lie in $\bl \cup \{0\}$. By continuity with respect to $\lambda$, there is a (simple) critical point $\cl$ which is the continuous extension of the critical point $\nu_0$  of $\jn$ in $\mathcal{A}^{*}(0)$. Each map has $n+d+1$ zeros, one of which, say $w_\la$, corresponds to the continuous extension of $w_0=a$. We now give a description of the position of the remaining $n+d$ critical points and the $n+d$ preimages of $z=0$ for $\jl$. 

\begin{lemma}
\thlabel{samoaracaii}
Let  $\jl$ satisfying conditions (a), (b), and (c). Assume $\xi \in \mathbb{C}$ is a $(n+d)$th-root of unity, $\xi^{n+d}=1$. Then, for $\modulu$ small enough, there exist $n+d$ free critical points, $c_{\la, \xi}$, and $n+d$ zeros, $z_{\la, \xi}$, given by 
$$ c_{\la, \xi}=\xi  \left( \frac{d Q(0) }{ -na }\right)^{\frac{1}{n+d}}\la^{\frac{1}{n+d}}+o\left( \la^{\frac{1}{n+d}}\right),$$ $$z_{\la, \xi}=\xi \left( \frac{ Q(0)}{a}\right) ^{\frac{1}{n+d}}\la^{\frac{1}{n+d}}+o\left( \lambda^{\frac{1}{n+d}}\right) .$$
\end{lemma}

\begin{proof}

Let us start with the zeros. Notice that all the zeros of $\jl$ (except for $w_{\la}$) must converge to $z=0$ when $\la$ tends to $0$. The zeros of $\jl(z)$ are the solutions of $$z^{n+d} (z -a)=-\lambda Q(z).$$

Since $a$ is away from $z=0$, there are $n+d$ zeros bifurcating from $z=0$, for $\modulu$ small enough. They are the fixed points of  $n+d$ operators $$T_{\la, \xi}(z)=\xi\left( \frac{Q(z)}{a-z}\right)^{\frac{1}{n+d}}\lambda^{\frac{1}{n+d}}=R(z)\lambda^{\frac{1}{n+d}},$$ where $\xi^{n+d}=1$ are roots of the unity. Observe that in a sufficiently small neighbourhood of $z=0$, $R(z)$ is holomorphic and bounded (notice that $R(0)\ne 0$), so $T_{\la, \xi}(z)\to 0$ as $\lambda \to 0$. We can approximate  $z_{\la, \xi}$ by $T_{\la, \xi}(0)=\xi\left( \frac{Q(0)}{a}\right) ^{\frac{1}{n+d}} \la^{\frac{1}{n+d}}$. Indeed, $$\left| z_{\la, \xi}-T_{\la, \xi}(0)\right| = \left| T_{\la, \xi}(z_{\la, \xi})-T_{\la, \xi}(0)\right| \le \sup\limits_{\omega \in [0, z_{\la, \xi}]}|T_{\la, \xi}'(\omega) | |z_{\la, \xi}-0|=|\la|^{\frac{1}{n+d}}\sup\limits_{\omega \in [0, z_{\la, \xi}]}|R'(\omega)|  |z_{\la, \xi}|.$$ 
For $\modulu$ small enough, there is no pole of $R'$ in a neighbourhood of $z=0$ containing the line segment $ [0, z_{\la, \xi}]$, so it is bounded by a constant, say $K_2$. It follows that $$\left|\frac{z_{\la, \xi}-\xi \left(\frac{Q(0)}{a}\right)^{\frac{1}{n+d}}\la^{\frac{1}{n+d}}}{\la^{\frac{1}{n+d}}}\right| \le K_2 |z_{\la, \xi}|.$$
Finally, since $\lim\limits_{\lambda \to 0}K_2 \left| z_{\la, \xi}\right|=0$,  it follows that $$z_{\la, \xi}=\xi  \left( \frac{ Q(0)}{a}\right)^{\frac{1}{n+d}}\la^{\frac{1}{n+d}}+o\left( \lambda^{\frac{1}{n+d}}\right) .$$
It can be shown analogously that the $n+d$ free critical points are solutions of the equation $$ \frac{1}{Q^2(z)}\left[ (n+1)z^nQ(z)-z^{n+1}Q'(z)-anz^{n-1}Q(z)+az^n Q'(z)\right]-\frac{\la d}{z^{d+1}}=0. $$ As before, we write the operators $S_{\la, \xi}$ as $$S_{\la, \xi}(z)=\xi \left(\frac{ d Q^2 (z)}{(n+1)zQ(z)-an Q(z) -z^2 Q'(z) +azQ'(z) }\right)^{\frac{1}{n+d}}\la^{\frac{1}{n+d}},$$ which have the critical points of $\jl$ as fixed points. The argument made is identical since $Q(0) \ne 0$, so $S_{\la, \xi}$ are holomorphic in the neighbourhood of $z=0$. Finally, the critical points of the perturbation map are  of the form $$c_{\la, \xi}=\xi \left(\frac{d Q(0) }{ -na }\right)^{\frac{1}{n+d}}\la^{\frac{1}{n+d}}+o\left(\lambda^{\frac{1}{n+d}}\right).$$

\end{proof} 

Next we show that there exists a straight annulus (we will show later that it belongs to a doubly connected Fatou component) which is mapped into $T_{\la}$ under $\jl$. Let 
$$
c_1=\frac{1}{2}\min\left\{ \left| \frac{d Q(0) }{ na }\right|^{\frac{1}{n+d}},\left|\frac{ Q(0)}{a}\right|^{\frac{1}{n+d}}\right\} \quad {\textrm {and}} \quad c_2=2\max\left\{ \left|\frac{d Q(0) }{ na }\right|^{\frac{1}{n+d}},\left|\frac{ Q(0)}{a}\right|^{\frac{1}{n+d}} \right\}.
$$ 

\begin{lemma}
\thlabel{cichicea}
Let $\jl$ satisfying conditions (a), (b), (c), and (d). Then, for $\modulu$ small enough, the straight annulus 
\begin{equation}
\label{eq:c1c2}
\Omega_\lambda:=A\left(\mathbb{S}_{c_1 |\la|^{\frac{1}{n+d}}}, \mathbb{S}_{c_2 |\la|^{\frac{1}{n+d}}}\right)
\end{equation}
contains the points $c_{\la, \xi}$ and $z_{\la, \xi}$ introduced in \thref{samoaracaii} and it is mapped into $\tl$ under $\jl$. 
\end{lemma}

\begin{proof}
The first part of the statement follows directly from the algebraic expression of the points $c_{\la, \xi}$ and $z_{\la, \xi}$ in \thref{samoaracaii}. The rest of the proof is devoted to show that  $\jl \left(\Omega_\lambda\right) \subset \tl$.

Let  $ m  = \min \{|z|, \, Q(z)=0\}$ and let  $M=\min\{ |Q(z)|,\, |z|<m/2 \}$ (notice that $M>0$ since $z=0$ is not a root of $Q$). Let $z \in \Omega_\lambda$. For $\modulu$ small enough we have $$|\jl(z)|< \frac{c_2^n |\la|^{\frac{n}{n+d}}(|a|+1)}{M}+\frac{ |\la|^{\frac{n}{n+d}}}{c_1^d}.$$ We can rewrite this as $|\jl(z)|<K_1|\la|^{\frac{n}{n+d}}$, where $K_1$ depends on $Q$, $c_1$ and $c_2$, but it does not depend on $z$ and $\lambda$.  By \thref{Klemma}, for $\modulu$ small enough, the disk centered at $z=0$ and of radius $K_1|\la|^{\frac{n}{n+d}}$ lies in $\tl$, as desired.

\end{proof}

In the next proposition we describe the skeleton of the dynamical plane for $\modulu$ {\it small} (compare  Figure \ref{pozanr1}). 
Recall that $w_{\la}$ is the zero of $\jl$ which is the continuous extension of $w_0=a$ and $\cl$ is the continuous extension of the critical point $\nu_0$ in $\mathcal{A}^{*}(0)$ of $\jn$. 

\begin{figure}

\includegraphics[width=7cm, height=7cm]{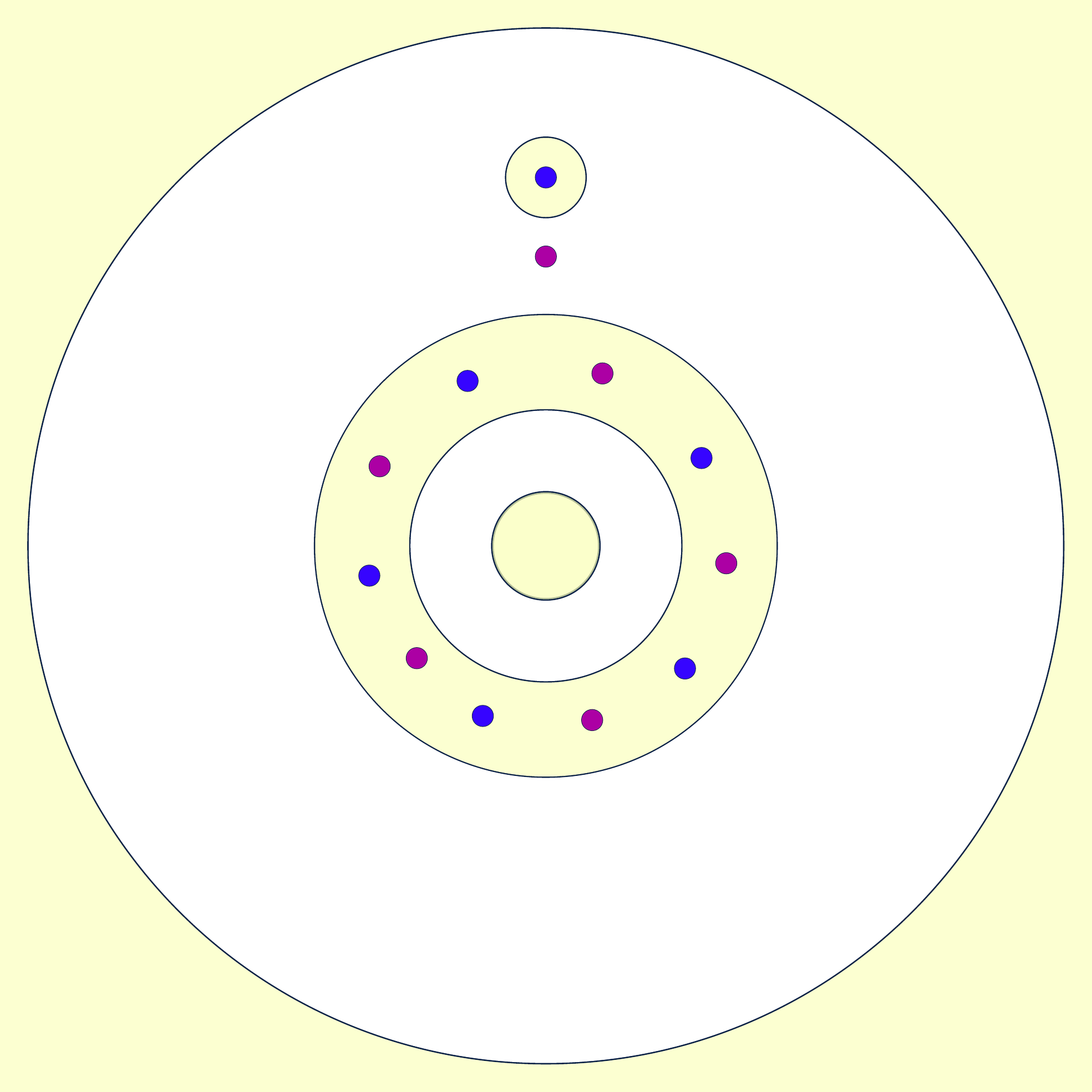}
\put(-105,96){$\tl$}
\put(-87,124){$\ala$}
\put(-92,163){$\dl$}
\put(-112,150){$\nu_{\la}$}
\put(-38,180){$\bl$}

  \caption{Partition of the dynamical plane with respect to $\bl$, $\ala$, $\tl$, and $\dl$, described in \thref{munnum}. Blue and purple points denote zeros and critical points, respectively.}
\label{pozanr1}

\end{figure}

\begin{prop}
\thlabel{munnum}
Let  $\jl$ satisfying conditions (a), (b), (c), and (d). Then, there exists a constant $\sefu=\sefu(a, Q,n,d)$ such that if $\lambda\neq 0$ and $\modulu<\sefu$ the following statements are satisfied:
\begin{enumerate}[label=(\roman*)]

\item{The Fatou component $\tl$ is simply connected and it is mapped with degree $d$ onto $\bl$ under $\jl$. There are no other preimages of $\bl$. }
\item{There exists a Fatou component $\ala$ which is doubly connected and contains exactly $n+d$ simple critical points, given by $c_{\la, \xi}$, and $n+d$ zeros, given by $z_{\la, \xi}$. Moreover, $\ala$ is mapped with degree $n+d$ onto $\tl$ and surrounds the origin. }
\item{Let $A^{\textrm {out}}$ be the annulus bounded by $\overline{\ala}$ and $\partial \bl$. There exists a Fatou component $\dl\subset A^{\textrm {out}}$ which is simply connected, is mapped with degree $1$ onto $\tl$, and contains $w_{\la}$.}
\item{The critical point $\cl$ lies in $A^{\textrm {out}}\setminus \dl$.}
\item{There are no preimages of $\tl$ other than $\dl$ and $\ala$.}
\item{Let $A^{\textrm {in} }$ be the annulus bounded by $\partial \tl$ and $\overline{ \ala}$. Then, $A^{\textrm {in} }$ is mapped onto the annulus $A(\partial \tl,\partial \bl)$ with degree $d$.}
\end{enumerate}
\end{prop}

\begin{proof}
Before proving the statements of the proposition we study the location and distribution of the critical points of $\jl$. 

By \thref{prop:continfty}, $\bl$ is simply connected (in the Riemann sphere) and it is mapped onto itself with degree $n+1$. Since the global degree of the map $\jl$ is $n+d+1$, there exist exactly $d$ preimages of $\infty$ outside $\bl$, counting multiplicity. Since $z=0$ is mapped to $\infty$ with degree $d$, there exist no other preimages of $\infty$  (different from the ones in $\bl$ and $z=0$). Moreover $\tl$ is mapped with degree $d$ onto $\bl$ (observe that up to this point we still do not know if $\tl$ is simply connected). 

Let $\Omega_\lambda$ be the annulus defined in \eqref{eq:c1c2}. By  \thref{cichicea}, we know that $\jl(\Omega_\lambda)\subset \tl$ and $\jl(\tl) \subset \bl$. Thus, $\Omega_\lambda \cap \tl =\emptyset$ and $\Omega_\lambda$ is part of a multiply connected Fatou component which is a preimage of $\tl$. We denote this Fatou component by $\ala$ (observe that up to this point we still do not know if  $\ala$ is doubly connected). 

We claim that $w_{\la}$ and $\cl$ do not belong to $\tl \cup \ala$. To see the claim we will prove that, for sufficiently small values of $\modulu$, $w_{\la}$ and $\cl$  belong to the annulus bounded by $\partial \bl$ and $\overline{\ala}$, denoted in what follows by $A^{\textrm {out}}$. Let $\gamma$ be a smooth Jordan curve which separates $z=0$  from $\nu_0$ and $w_0$, and such that $\jn^2(\gamma)$ is a Jordan curve that surrounds $z=0$ and lies in $\textrm{Int}(\gamma)$. Its existence follows from the B\"otcher coordinates of the fixed point $z=0$ for the unperturbed map. Notice that, by construction, $\gamma$  does not depend on $\lambda$ and it has a definite distance to $z=0$. By continuity with respect to $\la$, for $\modulu$ small enough, $\jl^2(\gamma)\subset\textrm{Int}(\gamma)$. Since $\jl^2(\ala) \subset \bl$ we conclude that $\gamma\cap \ala=\emptyset$. Shrinking $\modulu$, if necessary, we claim that $\gamma$ lies outside $\textrm{Bdd}\left(\Omega_\lambda\right)$. Indeed,  according to \eqref{eq:c1c2} the annulus $\Omega_\lambda$ collapses to the origin as $\lambda \to 0$ while $\gamma$ keeps in a definite distance to $z=0$. Finally, notice that for  $\modulu$ small $\cl$ and $w_{\la}$ remain as close as we want to $\nu_0$ and $w_0$, respectively. Therefore, $\gamma$ separates $\cl$ and $w_{\la}$ from $\tl$ and $\ala$.  Let $\mathcal C$ be a constant such that if $\modulu < \mathcal C$ all the above is true. Now we are ready to prove the statements. 

Since $\tl$ contains only one critical point at $z=0$ with multiplicity $d-1$ and it is mapped with degree $d$ onto the topological disk $\bl$, it follows from the Riemann-Hurwitz formula that $\tl$ is simply connected. This proves \textit{(i)}. Similarly, $\ala$ contains exactly $n+d$ simple critical points and it is mapped with degree $n+d$ onto the topological disk $\tl$. Thus, it is doubly connected by Riemann-Hurwitz formula. This proves \textit{(ii)}.

The point $w_{\la}$ is a preimage of $z=0$ which lies in $A^{\textrm {out}}$, so it must belong to a Fatou component, denoted by $D_\lambda$, different from $T_\lambda$ and $A_\lambda$. Moreover, $\jl\left(D_\lambda\right)=T_\lambda$. Since $w_{\la}$ is the only (simple) preimage of $z=0$ in $D_\lambda$ we conclude that $\dl$ is mapped with degree $1$ onto $\tl$ and is a conformal copy of $\tl$. In particular, $\cl \notin \dl$ and all preimages of $z=0$ belong to either $A_\lambda$ and $\dl$. This proves \textit{(iii)}, \textit{(iv)}, and \textit{(v)}.

Finally, to prove statement \textit{(vi)} we just notice that 
$\jl |_{A^{\textrm {in} }} : A^{\textrm {in} }\rightarrow A(\partial \tl,\partial \bl)$ is a proper map. Since its degree is accomplished on the boundaries and $\partial \tl$ is mapped onto $\partial \bl$ with degree $d$, it follows that $A^{\textrm {in} }$ is mapped onto its image with degree $d$.

\end{proof}

We now prove that $\jl$ is conjugate to a finite Blaschke product on the annulus $A^{\textrm {out}}$ introduced in the previous proposition.

\begin{prop}
\thlabel{prop:conjblas}
Let $\jl$ satisfying (a), (b), (c), and (d). Let $\lambda\neq 0$, $\modulu<\sefu$. Then, there exist an analytic Jordan curve $\Gamma \subset \ala$ which surrounds $z=0$, $b \in \mathbb{D}^{*}$, $\theta\in [0,1)$, and
a quasiconformal map $\varphi: \hat{\mathbb{C}} \to \hat{\mathbb{C}} $ such that $\varphi \circ R_{b,\theta} = \varphi \circ \jl$ on 
$A(\Gamma, \, \partial \bl)$, where $$
R_{b,\theta}=e^{2 \pi i \theta} z^n \frac{z-b }{1-\overline{b}z}$$
is a Blaschke product.
\end{prop}

\begin{proof}
We claim that there exists an analytic Jordan curve $\Gamma \subset A_\lambda$ which surrounds $z=0$ and the $n+d$ critical points $c_{\la, \xi}$ and which is mapped with degree $n$ to its image under $\jl$. 

To see the claim let $\gamma$ be an analytic Jordan curve in the interior of $\tl$ surrounding $z=0$ and the $n+d$ critical values, images of the critical points $c_{\la, \xi}$. Clearly, the annulus $A(\gamma , \, \partial \tl)$ contains no critical values and, from the Riemann-Hurwitz formula (compare \thref{corocoro}),  any connected component of its preimage is an annulus bounded by preimages of $\gamma$ and $\partial \tl$. It follows from \thref{munnum} that two (among a total of three) of those preimages are disjoint annuli in $A_\lambda$, one associated to the internal boundary of $A_\lambda$ and another associated to the external one. Denote them by $\mathcal G^{\textrm {in} }$ and $\mathcal G^{\textrm {out}}$. By construction, $\jl$ restricted to those two preimages is a proper map. We know that $\jl$ restricted to $\ala$ is proper of degree $n+d$ (see \thref{munnum} \textit{(ii)}) while $\jl$
restricted to $\mathcal G^{\textrm {in}}$ is proper of degree $d$ (notice that the degree is achieved in the boundary, compare with \thref{munnum}\textit{(vi)}). All together implies that $\jl$ restricted to $\mathcal G^{\textrm {out}}$ is proper of degree $n$. Let $\Gamma$ be the inner boundary of $\mathcal G^{\textrm {out}}$. Then, $\Gamma$ is an analytic Jordan curve, it maps to $\gamma$ with degree $n$, and it surrounds the origin as well as all critical points $c_{\la, \xi}$, as desired.

The remaining part of the proof is analogous to the one of \cite[Proposition 3.1]{Can2}, so we provide the main idea and leave the details to the reader. The strategy is to use a similar construction to the one of the Straightening Theorem for polynomial-like mappings (compare \cite[Theorem 7.4]{BF}) to glue a dynamics conjugated to the one of the map $z \to z^n$ inside the curve $\gamma$, keep $\jl$ outside $\Gamma$, and interpolate using a quasi-conformal map in the annulus $A\left(\gamma,\Gamma\right)$. 
In this way we obtain a quasiregular map $F$ of the Riemann sphere which has $z=0$ as superattracting fixed point of local degree $n$ ($F$ is actually holomorphic around $ z=0$). The map $F$ coincides with $\jl$ on $\textrm{Ext}(\Gamma)$, all points in $\textrm{Int}(\partial \bl)$ converge to $z=0$ under iteration of $F$, and it maps $\textrm{Int}(\partial \bl)$ onto itself with degree $n+1$ (since we have that $z=0$ maps to itself with degree $n$ and $w_\lambda$ is the only further preimage of $z=0$). 

The map $F$ is conjugate to a holomorphic function $f$ via a quasiconformal $\varphi$ map fixing $z=0$. The basin of attraction of $z=0$ under $f$ is given by $\varphi (\textrm{Int}(\partial \bl))$. Since the basin of attraction is simply connected, $f$ is conjugate to a Blaschke product in  $\varphi (\textrm{Int}(\partial \bl))$. Since $z=0$ is superattracting of local degree $n$ and $f$ maps $\varphi (\textrm{Int}(\partial \bl))$ onto itself with degree $n+1$, the Blaschke product has the form  $R_{b,\theta}=e^{2 \pi i \theta} z^n \frac{z-b }{1-\overline{b}z}$, where $b \in \mathbb{D}^{*}$ and $\theta$ satisfies $|e^{2 \pi i\theta}|=1$. Since $F$ coincides with $\jl$ in $A(\Gamma, \, \partial \bl)$, it follows that $\jl$ is conjugate to $R_{b,\theta}$ in $ A(\Gamma,  \, \partial \bl)$.
\end{proof}

It will be crucial in what follows to have a deep understanding of the preimages of curves which surround the origin $z=0$ (as well as Fatou components). The following proposition describes this in a precise way.

\begin{prop}

\thlabel{prop:preimcurves}
Let $\jl$ satisfying (a), (b), (c), and (d). Let $\lambda\neq 0$, $\modulu<\sefu$. Let $\gamma \subset A(\partial \tl,\partial \bl)$ be  a Jordan curve which surrounds $z=0$. Then, $\jl^{-1}(\gamma)$ contains a single connected component in $A^{\textrm {in} }$, which surrounds $z=0$ and is mapped with degree $d$ onto $\gamma$. The other components of $\jl^{-1}(\gamma)$ lie in $A^{\textrm {out}}$ and, depending on the location of the free critical value, i.e. $\jl(\cl)$, one of the following holds:

\begin{enumerate}[label=(\roman*)]

\item{If $\jl (\cl)\in \textrm{Int} (\gamma)$, then $\jl^{-1} (\gamma)$ has a single connected component in $A^{\textrm {out}}$. Indeed, it is a Jordan curve which surrounds $z=0$ and it is mapped with degree $n+1$ onto $\gamma$ under $\jl$.}
\item{If $\jl(\cl)\in \gamma$, then $\jl^{-1} (\gamma)$ has a single connected component in $A^{\textrm {out}}$ consisting of 2 Jordan curves intersecting precisely at $\cl$. One is a Jordan curve $\gamma_0^{-1}$ which surrounds $z=0$, but not $w_{\la}$, and it is mapped with degree $n$ onto $\gamma$. The other is a Jordan curve $\gamma_w^{-1}$ which surrounds $w_{\la}$, but not $z=0$, and it is mapped with degree $1$ onto $\gamma$. The curve $\gamma_0^{-1}$ does not surround $\gamma_w^{-1}$.}
\item {If $\jl(\cl)\in \textrm{Ext} (\gamma)$, then $\jl^{-1} (\gamma)$ has 2 disjoint components in $A^{\textrm {out}}$. One is a Jordan curve $\gamma_0^{-1}$ which surrounds $z=0$, but not $w_{\la}$, and it is mapped with degree $n$ onto $\gamma$. The other is a Jordan curve $\gamma_w^{-1}$ which surrounds $w_{\la}$, but not $z=0$, and it is mapped with degree $1$ onto $\gamma$. The curve $\gamma_0^{-1}$ does not surround $\gamma_w^{-1}$.}

\end{enumerate}
\end{prop}

\begin{proof}
We first notice that given any Jordan curve in $ A(\partial \tl,\partial \bl)$ 
all preimages should be located either in $A^{\textrm {in} }$ or $A^{\textrm {out}}$ since $T_\lambda$, $\ala$ and $\bl$ are mapped outside  $A(\partial \tl,\partial \bl)$. Moreover, by \thref{munnum}\textit{(vi)} any Jordan curve in $ A(\partial \tl,\partial \bl)$ should have  preimage(s) in $A^{\textrm {in} }$ as well as in $A^{\textrm {out}}$. 

Let $\gamma \in  A(\partial \tl,\partial \bl)$ be a Jordan curve surrounding the origin. First we study the topology of the preimage(s) of $\gamma$ in $A^{\textrm {in} }$. By \thref{munnum}, $\gamma$ has exactly $d$ preimages in $A^{\textrm {in} }$.  Let $\gamma_0$ be one of the preimages of $\gamma$ in $A^{\textrm {in} }$. The goal is to show that in fact $\gamma_0$ is mapped by $\tl$ with degree $d$, so there are no other preimages whatsoever. Observe that 
${\textrm Int\left(\gamma_0\right)}$ should contain either a pole, a zero, or $z=0$, otherwise it cannot be mapped to $\gamma$ which surrounds $z=0$. Therefore, $\gamma_0$  surrounds $z=0$. Take the annulus  $A(\gamma, \partial \bl)$ and consider its preimage in $A^{\textrm {in} }$. Since the only preimage of  $\partial \bl$ in $\overline{A^{\textrm {in} }}$ is $\partial \tl$ and $A^{\textrm {in} }$ contains no critical point, the preimage should be the annulus $A(\partial \tl, \gamma_0)$. The map $\jl|_{A(\partial \tl, \gamma_0)}$ is proper of degree $d$ since $\jl$ maps $\partial \tl$ onto  $\partial \bl$ with degree $d$. We conclude that $\gamma_0$ is mapped with degree $d$ onto $\gamma$,  as desired.

The proof of statements \textit{(i)}-\textit{(iii)} about the topology of the preimage(s) of $\gamma$ in $A^{\textrm {out}}$ is analogous to the one of \cite[Proposition 3.3]{Can2} by using that $\jl$ is conjugate to the Blaschke product $R_{\theta, b}$ in $A^{\textrm {out}}$ (see \thref{prop:conjblas}).

\end{proof}

\begin{remark}
It follows from \thref{prop:preimcurves} that each Fatou component different from $\tl$ and $\bl$ which surrounds $z=0$ (and so it contains a Jordan curve which surrounds $z=0$) has exactly two preimages which also surround $z=0$. One of them lies in $A^{\textrm {in} }$ and the other lies in $A^{\textrm {out}}$. 
\end{remark}

The following lemma shows that Fatou components which do not surround the origin do not have preimages which surround it. 

\begin{lemma}
\thlabel{nusurround}

Let $\jl$ satisfying (a), (b), (c), and (d). Let $\lambda\neq 0$, $\modulu<\sefu$. Let $U$ be a multiply connected Fatou component. If $U$ does not surround $z=0$, then no component of its preimage $\jl^{-1}(U)$ surrounds $z=0$. 
\end{lemma}

\begin{proof}

Suppose that $U$ does not surround $z=0$ and let $V$ be a preimage of $U$ which surrounds $z=0$. Let $U'=\textrm{Bdd}(U)$ and $V'$ the preimage of $U'$ which contains $V$. Observe that $U'\subset A(\partial \tl,\partial \bl)$. Since $V'\subset A^{\textrm {in} } \cup A^{\textrm {out}}$, it can contain at most one critical point. Since $\jl|_{V'}: V'\rightarrow U'$ is proper and V' contains at most one critical point, it follows from the Riemann Hurwitz formula that $V'$ is simply connected (compare \thref{corocoro}). Since $V$ surrounds the origin, then $z=0$ lies in $V'$. However, this is impossible since $z=0$ is mapped to $\infty$ and $U'$ is bounded.

\end{proof}

\thref{munnum} tells us that for $\modulu$ small enough the map $\jl$ is essentially uni-critical since all critical points except $\cl$ belong to $A_{\la}(\infty)$. Up to now, however, we have not imposed any particular dynamical behaviour for $\cl$. With the aim of proving the main results of this paper from now on we restrict ourselves to parameters for which the free critical point $\cl$ belongs to $A_{\la}(\infty)$ (sometimes the term {\it captured parameters} is used). 

Under this assumption, \thref{munnum} implies that $\nu_\lambda \in \mathcal{A}_{\la}(\infty) \setminus \left( \bl\cup \tl\cup\dl\cup\ala\right)$. We further assume that $\cl$ belongs to a Fatou component which is an eventual preimage of $\ala$ that surrounds $z=0$. The following result gives relevant notation and determines a partition of the dynamical plane (compare Figure \ref{pozanr2}) that will be extremely  useful to study achievable connectivities of Fatou components.

\begin{theorem}
Let $\jl$ satisfying (a), (b), (c), and (d). Let $\lambda\neq 0$, $\modulu<\sefu$. Assume that $\cl \in \uc$, where $\uc$ is a Fatou component which is eventually mapped onto $\ala$ and surrounds $z=0$. Then, $\uc$ is triply connected and $\uc \subset A^{\textrm {out}}$. Moreover, the following statements hold. 

\begin{enumerate}[label=(\roman*)]
\item{The set $\overline{\uc}$ bounds an open disk $\mathcal{U}_1$ which is mapped with degree $1$ onto the open disk $\mathcal{V}_n \cup \overline{\tl}$, where $\mathcal{V}_n$ is the annulus bounded by $\partial \tl$ and $\overline{\jl(\uc)}$. In particular, $w_\lambda \in \mathcal{U}_1$.}
\item{The annulus $\mathcal{U}_{n+1}$ bounded by $\overline{\uc}$ and $\partial \bl$ is mapped with degree $n+1$ onto the annulus $\mathcal{V}_{n+1}$ bounded by  
$\overline{\jl(\uc)}$ and $\partial \bl$.}
\item{The annulus $\mathcal{U}_n$ bounded by $\overline{\ala}$  and $\overline{\uc}$ is mapped with degree $n$ onto the annulus $\mathcal{V}_n$ bounded by $\partial \tl$ and $\overline{\jl(\uc)}$.}
\item{The annulus $\mathcal{U}_d$ bounded by $\partial \tl$ and $\overline{\ala}$ (i.e., the annulus $A^{\textrm {in} }$)  is mapped with degree $d$ onto the annulus $A(\partial \tl,\partial \bl)$.}
\item{The Fatou component $\uc$ lies in $\mathcal{V}_{n+1}$ and it is mapped under $\jl$ with degree $n+1$ onto its image. In particular, $\uc$ surrounds $\jl(\uc)$.}
\end{enumerate}
\thlabel{dumnezeu}

\end{theorem}

\begin{figure}
\hfill
\subfigure[The partition with respect to $\uc$ and $\ala$.]{\includegraphics[width=7cm]{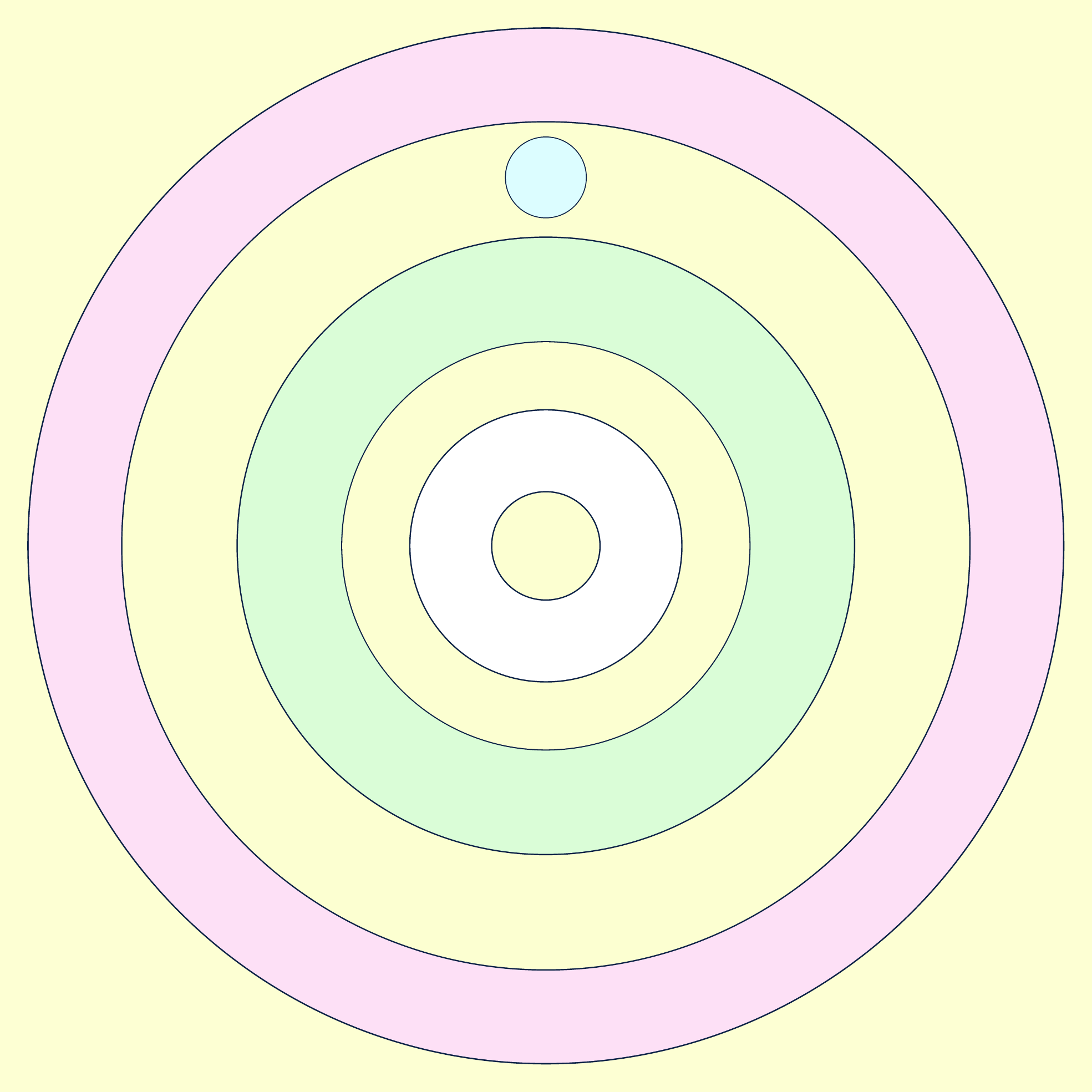}}
\put(-105,96){$\tl$}
\put(-152,144){$\uc$}
\put(-129,119){$\ala$}
\put(-105,163){$\mathcal{U}_{1}$}
\put(-38,180){$\bl$}
\put(-22,96){$\mathcal{U}_{n+1}$}
\put(-58,96){$\mathcal{U}_{n}$}
\put(-88,96){$\mathcal{U}_{d}$}
\hfill
\subfigure[The partition with respect to $\jl(\uc)$.]{\includegraphics[width=7cm]{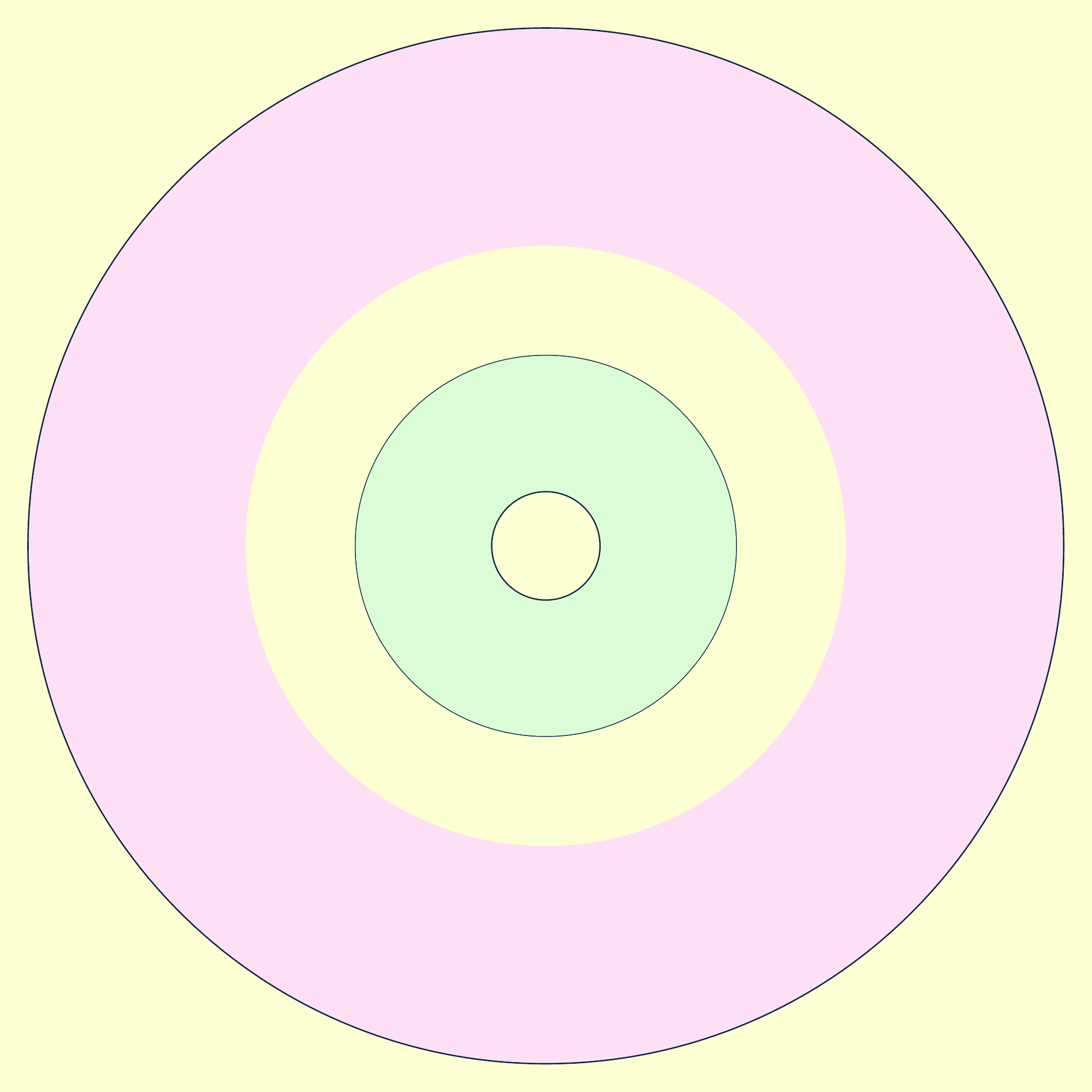}}
\put(-105,96){$\tl$}
\put(-118,137){$\jl(\uc)$}
\put(-38,180){$\bl$}
\put(-180,96){$\mathcal{V}_{n+1}$}
\put(-130,96){$\mathcal{V}_{n}$}
\hfill
\caption{Partitions of the dynamical plane introduced in \thref{dumnezeu}.}
\label{pozanr2}
\end{figure}
\begin{remark}
Statement \textit{(iv)} in \thref{dumnezeu} coincides with statement \textit{(vi)} of \thref{munnum}. We use this double naming ($\mathcal{U}_d$ and $A^{\textrm {in} }$) in order to uniformize notation in what follows. Notice that every set $\mathcal{U}_i$, $i=d, \, n,\, n+1$, is mapped onto its image with degree $i$. This notation is particularly useful in Section \ref{sec:Theorem_AB}. Also, notice that in order to simplify notation we avoid indicating the dependence of $\uc$ and $\mathcal{U}_i,\,i=1, \, d, \, n,\, n+1$, with respect to the parameter $\lambda$.  
\end{remark}

\begin{proof}

By the Riemann-Hurwitz formula, the iterated preimages of $\ala$ are doubly connected unless they contain a critical point. Under our hypothesis, this occurs precisely at $\uc$ (since the only critical point eventually mapped in $A_\lambda$ is $\nu_\lambda$). A direct application of the Riemann-Hurwitz formula implies that, since $\nu_\lambda$ is a simple critical point, $\uc$ is triply connected.  Moreover, $\uc\subset A^{\textrm {out}}$ since $\nu_\lambda \in A^{\textrm {out}}$. This proves the first part of the statement.

From above $\partial \uc$ has three components. Since $\uc$ separates $z=0$ from $z=\infty$, there should be exactly two components of $\partial \uc$ surrounding $z=0$. We denote them by $\gamma_c^{\textrm{in}}$ and $\gamma_c^{\textrm {out}}$, where $\gamma_c^{\textrm{in}}\subset \textrm{Int}(\gamma_c^{\textrm {out}})$. The other component of $\partial \uc$, denoted by $\gamma_c^1$, does not surround $z=0$. 

Set $\mathcal{U}_1=\textrm{Int}(\gamma_c^1)$. Since $\gamma_c^1$ is mapped onto a component of $\partial \jl(\uc)$, $\mathcal{U}_1$ is mapped either to the bounded or the unbounded component of the complement of $\jl(\uc)$ (which is an annulus by hypothesis). However, since all  poles are in $\bl\cup\tl$, then $\mathcal{U}_1$ should be mapped onto the bounded component of  $\jl(\uc)$. Therefore, $\mathcal{U}_1$ contains the zero $w_{\la}$ (and no other preimages of $z=0$). We conclude that $\jl|_{\mathcal{U}_1}$ has degree $1$. In particular, $\gamma_1$ is mapped onto its image with degree $1$. This proves \textit{(i)}.

Let $\mathcal{U}_{n+1}$ be the annulus bounded by $\overline{\uc}$ and $\partial \bl$,  and let $\mathcal{V}_{n+1}=\jl (\mathcal{U}_{n+1})$. By construction, $\mathcal{V}_{n+1}$ is the annulus bounded by $\overline{\jl(\uc)}$ and $\partial \bl$. It is immediate that the map $\jl|_{\mathcal{U}_{n+1}}:\mathcal{U}_{n+1}\rightarrow \mathcal{V}_{n+1}$ is proper. Since the degree is accomplished on the boundaries and $\partial \bl$ is mapped onto itself with degree $n+1$, $\mathcal{U}_{n+1}$ is mapped onto $\mathcal{V}_{n+1}$ with degree $n+1$. This proves \textit{(ii)}. The proof of statement \textit{(iii)} is similar and statement \textit{(iv)} was already proven in \thref{munnum}.  


Finally, we prove statement \textit{(v)} by contradiction. Assume that $\uc$ does not lie in $\mathcal{V}_{n+1}$. Then, either $\uc$ maps onto itself (which is impossible) or $\mathcal{V}_{n+1}\subset \mathcal{U}_{n+1}$. This would imply that $\mathcal{U}_{n+1}$ is mapped under $\jl$ on itself and, hence, there exists a periodic Fatou component different from $\mathcal{A}_{\la}(\infty)$. This is impossible since, by assumption, the orbits of all critical points converge to $z=\infty$.
\end{proof}

\begin{remark}
\thlabel{haiininterior}
Under the assumptions of \thref{dumnezeu}, if $U$ is a Fatou component  which surrounds $z=0$ and lies in $\mathcal{U}_{n+1}$ or $\mathcal{U}_n$, then it follows from \thref{prop:conjblas} that its image lies in $\textrm{Bdd}(U)$. Indeed, $\mathcal{U}_n$ and $\mathcal{U}_{n+1}$ are contained in $A^{\textrm {out}}$ (see Figure \ref{pozanr2}), and $A_{\textrm {out}}$ belongs to the region where the dynamics are conjugate to the ones of a Blaschke product. 

\end{remark}

As it will become clear in the next sections devoted to prove the main results of this paper, the presence of Fatou components with high connectivity in the dynamical plane is based on taking special iterated preimages of $\uc$. With this in mind we end the section by stating the following corollary of \thref{prop:preimcurves} and \thref{dumnezeu}.

\begin{corollary}
\thlabel{chichicea}
Let $\jl$ satisfying (a), (b), (c), and (d). Let $\lambda\neq 0$, $\modulu<\sefu$. Assume that $\cl \in \uc$, where $\uc$ is an iterated preimage of $\ala$ which surrounds $z=0$. Let $W$ be a Fatou component which surrounds $z=0$, different from $\tl$ and $\ala$. Then, the following statements hold.
\begin{enumerate}[label=(\roman*)]
\item If $W \subset \mathcal{V}_{n+1}$, then it has a unique preimage in $\mathcal{U}_d$, which surrounds $z=0$ and is mapped under $\jl$ onto $W$ with degree $d$, and a unique preimage in  $\mathcal{U}_{n+1}$, which surrounds $z=0$ and is mapped under $\jl$ onto $W$ with degree $n+1$.
\item If $W \subset \mathcal{V}_{n}$, then it has a unique preimage in $\mathcal{U}_d$, which surrounds $z=0$ and is mapped under $\jl$ onto $W$ with degree $d$, and two further preimages. One lies in $\mathcal{U}_{n}$, surrounds $z=0$ and is mapped under $\jl$ onto $W$ with degree $n$. The other one lies in $\mathcal{U}_1$,  does not surround $z=0$, and is mapped under $\jl$ onto $W$ with degree $1$.
\end{enumerate}
\end{corollary}

\section{Proofs of theorems A and B}\label{sec:Theorem_AB}
In this section we prove Theorem A and Theorem B. We first show that Fatou components which do not surround $z=0$ cannot be used to achieve higher connectivities.
\begin{lemma}
\thlabel{nusurrconstant}
Let $\jl$ satisfying (a), (b), (c), and (d). Let $\lambda\neq 0$, $\modulu<\sefu$. Assume also that $\cl \in \uc$. Let $V$ be a Fatou component which does not surround $z=0$. Then $V$ and all of its eventual preimages have the same connectivity.
\end{lemma}
\begin{proof} 
Let $V$ be a Fatou component which does not surround $z=0$ and let $U$ be a preimage of it. By \thref{nusurround}, $U$  does neither surround $z=0$. It follows that $\textrm{Bdd}(U)$ does not contain any critical point. Therefore, the map $\jl|_{\textrm{Bdd}(U)} : \textrm{Bdd}(U)\rightarrow \textrm{Bdd}(V)$ is a proper map of degree $1$. We can conclude that $\jl|_{U} :U \rightarrow V$ is conformal, so $U$ and $V$ have the same connectivity.
\end{proof}
Next we give the form of the connectivities of Fatou components of $\jl$. We want to remark that not all these connectivities are achievable (see Theorem A).
\begin{prop}
\thlabel{bambampocpoc}
Let $\jl$ satisfying (a), (b), (c), and (d). Let $\lambda\neq 0$, $\modulu<\sefu$. Assume also that $\cl \in \uc$, where $\uc$ is an iterated preimage of $\ala$ which surrounds $z=0$. All Fatou components have connectivity 1, 2, or  $\kappa=(n+1)^{i}n^j d^\ell+2$, where $i,j,\ell \in \mathbb{N}$.
\end{prop}

\begin{proof}

By \thref{munnum}, $\dl$ and all its eventual preimages have connectivity $1$ since none of them can contain critical points. Analogously, all eventual preimages of $\ala$ other than $\uc$ and its preimages have connectivity $2$ since none of them contain critical points (see \thref{corocoro}). Finally, we study the connectivity of the preimages of $\uc$. By \thref{nusurrconstant}, it suffices to study preimages of $\uc$ which surround $z=0$. It follows from the Riemann-Hurwitz formula that if $f:U\rightarrow V$ is proper of degree $q$ without critical points and $V$ has connectivity $p+2$, then $U$ has connectivity $qp+2$. By \thref{chichicea}, all preimages of $\uc$ which surround $z=0$ map to their images with degree $d, \, n$, or $n+1$. Since $\uc$ has connectivity $3=1+2$, the possible connectivities of the preimages of $\uc$ surrounding $z=0$ are of the form $\kappa=(n+1)^{i}n^j d^\ell+2$, where $i,j,\ell \in \mathbb{N}$.

\end{proof}

According to \thref{chichicea}, if $U$ is a iterative preimage of $\uc$ which surrounds $z=0$, its preimages which surround $z=0$ may be located in $\mathcal{U}_d$, $\mathcal{U}_n$ or $\mathcal{U}_{n+1}$. Next lemma shows that there are achievable upper bounds for the itineraries of iterated preimages of $\uc$. Recall from \thref{haiininterior} that if $U \subset \mathcal{U}_n$ is an iterated preimage of $\uc$, then either $\jl(U) \subset \mathcal{U}_n$ or $\jl (U) \subset \mathcal{U}_d$. Let $k\ge1$ be minimal such that $\jl^{k}(\uc)\subset \textrm{Bdd}(\ala) $. The first half of \thref{maxitaxi} shows that if $U \subset \mathcal{U}_n $ is a preimage of $\uc$ which surrounds $z=0$, then the itinerary of $U$ intersects $\mathcal{U}_d$ in $p \le k$ iterations. The second half of \thref{maxitaxi} shows that if $U \subset \mathcal{U}_d $ is a preimage of $\uc$ which surrounds $z=0$, then there exist at least $k-1$ consecutive backwards iterates of $U$ which lie in $\mathcal{U}_n$.
\begin{lemma}
\thlabel{maxitaxi}

Let $\jl$ satisfying (a), (b), (c), and (d). Let $\lambda\neq 0$, $\modulu<\sefu$. Assume also that $\cl \in \uc$. Let $k\ge1$ such that $\jl^{k}(\uc)\subset \textrm{Bdd}(\ala) $ and $\jl^j (\mathcal{U}_c) \subset \mathcal{U}_n$ for $1\le j <k$.
\begin{enumerate}[label=(\roman*)]

\item{Let $U\subset \mathcal{U}_n$ be an iterated preimage of $\uc$ which surrounds $z=0$. Then, there exists $ p \ge 1$ such that  $S^p(U) \subset \mathcal{U}_d$ and $S^{p'}(U)\in \mathcal{U}_n$ for $0\leq p'<p$. Moreover, $p$ satisfies $p\leq k$.}
\item{Let $U\subset \mathcal{U}_d$ be a preimage of $\uc$ which surrounds $z=0$. Then, there exists $U'\subset \mathcal{U}_n$ such that $\jl^j(U') \subset \mathcal{U}_n$ for $0 \le j<k-1$ and $\jl^{k-1}(U')=U$.}
\end{enumerate}

\end{lemma}

\begin{proof}

\begin{figure}
\includegraphics[width=7cm, height=7cm]{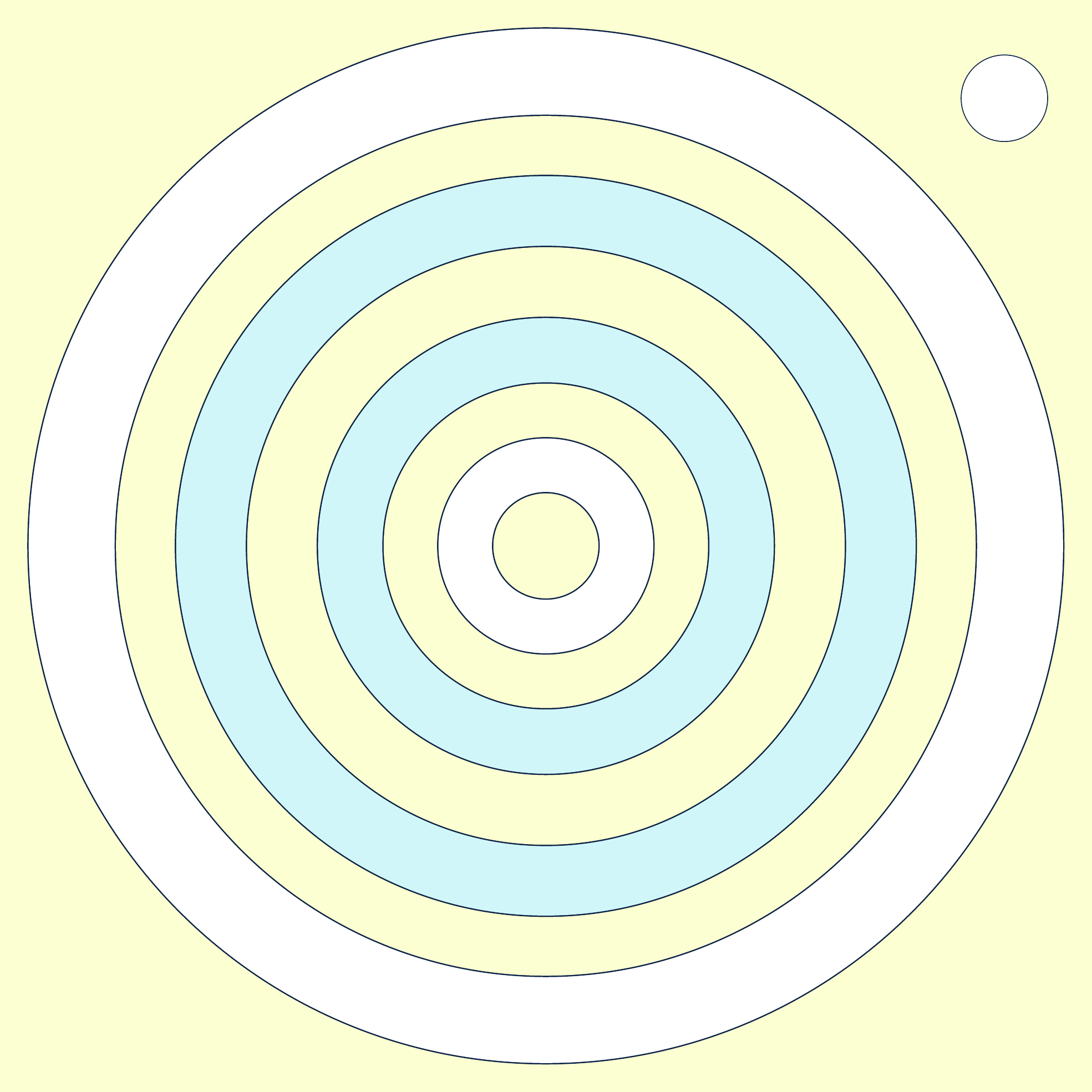}
\put(-180, 169) {$\uc$}
\put(-22,178) {$\mathcal{U}_1$}
\put(-104,98) {$\tl$}
\put(-121, 113) {$W_2$}
\put(-105, 182) {$B_1$}
\put(-106, 156) {$B_{2}^{ \textrm out}$}
\put(-106, 131) {$B_{2}^{\textrm in}$}
\put(-159, 145) {$W_1$}
\put(-144, 128) {$\ala$}
\caption{Description of the situation in the proof of \thref{maxitaxi}, where $k=2$ and $W_2 \subset \mathcal{U}_d$. In this case, $B_2=B_{2}^{\textrm {out}} \cup  \overline{\ala} \cup B_{2}^{\textrm{in}}$.}
\label{pozanr3} 
\end{figure}
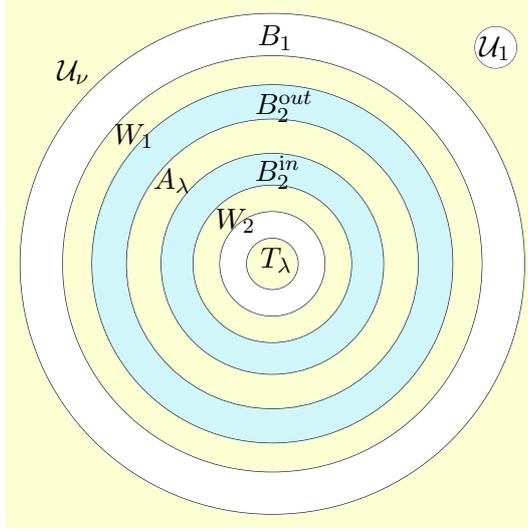

Set $W_{i}=\jl^{i}(\uc), \, i=0,\ldots,k$. By \thref{haiininterior},  $W_i$ surrounds $W_{i+1}$, $i=0,\ldots,k-1$. Let $B_{i+1}$ be the annulus bounded by $\overline{W_i}$ and $\overline{W_{i+1}}$, $i=0,\ldots, k-1$. It follows that $\jl(B_i)=B_{i+1}$, $i=1,\ldots, k-1$. Observe that if $W_k \subset\mathcal{U}_d$, then $B_k= B_{k}^{\textrm out}\cup \overline{\ala} \cup B_{k}^{\textrm in}$, where $B_{k}^{\textrm out}= B_k \cap \mathcal{U}_n $ and $B_{k}^{\textrm in}= B_k \cap \mathcal{U}_d$ (see  Figure \ref{pozanr3}). Along the proof we distinguish the cases $W_k=\ala$ and $W_k \subset \mathcal{U}_d$. 

We now prove statement \textit{(i)}. Assume first that $W_k=\ala$. Let $U \subset B_i$, $ i=1,\ldots,k$, be an eventual preimage of $\uc$ which surrounds $z=0$. Then $\jl^{k+1-i}(U)\subset \mathcal{U}_d$. We can conclude that if $U \subset \mathcal{U}_n$ is a preimage of $\uc$ which surrounds $z=0$, then there exists $p \leq k$ such that $\jl^{p}(U) \subset \mathcal{U}_d$. In fact, $p=k-i$ if $U \subset B_i$. Now assume $W_k \subset \mathcal{U}_d$. For $k=1$, $\jl(\mathcal{U}_n) \subset \mathcal{U}_d$ and the conclusion follows. For $k \ge 2$ we have $W_1 \subset \mathcal{U}_n$ (so $B_{k-1}$ and $B_{k}$ exist). Let $U \subset B_i$, $i=1, \dots, k-1$, be a preimage of $\uc $ which surrounds $z=0$ and let $V= \jl^{k-1-i}(U)$. Observe that $V \subset B_{k-1}$ is an eventual preimage of $\uc$ which surrounds $z=0$ and $\mathcal{U}_d$ is the disjoint union of $B_{k}^{\textrm in}$, $\overline{W_k}$, and $\jl(B_{k}^{\textrm out})$. Then, $\jl(V) \subset B_{k}^{\textrm out}$ (and $\jl^2(V)\subset \mathcal{U}_d$) or $\jl(V) \subset B_{k}^{\textrm in}$ (and $\jl(V)\subset \mathcal{U}_d$). Since $\jl^{k-1-i}(B_i)=B_{k-1}$, $i=0,\ldots, k-1$, this concludes the proof of statement \textit{(i)}.

Now we prove \textit{(ii)}. Let $U\subset \mathcal{U}_d$ be a preimage of $\uc$ which surrounds $z=0$. Assume first that $W_k=\ala$. Then $\mathcal{U}_d=\jl(B_k)=\jl^k (B_1)$. So there exists $U'\subset \mathcal{U}_1$ such that $\jl^j(U') \subset B_{j+1} \subset \mathcal{U}_n$ for $0 \le j<k$ and $\jl^{k}(U')=U$. Now let $W_k \subset \mathcal{U}_d$. For $k=1$ there is nothing to prove. For $k \ge 2$ we have $W_1 \subset \mathcal{U}_n$ (so $B_{k-1}$ and $B_{k}$ exist). Moreover, $\mathcal{U}_d = B_{k}^{\textrm in} \cup \overline{W_k}\cup \jl(B_{k}^{\textrm out})$. Since $U$ is a preimage of $\uc$ we have that $U \ne W_k$. We distinguish 2 cases. If $U \subset B_{k}^{\textrm in}$, then $U\subset \jl(B_{k-1})$ and we can take preimages through the sets $B_i$ so that there exists $U'\subset B_1$ such that $\jl^j(U') \subset B_{j+1}\subset \mathcal{U}_n$ for $0 \le j<k-1$ and $\jl^{k-1}(U')=U$. Finally, if $U \subset \jl(B_{k}^{\textrm out})$, then $U\subset \jl(B_{k})$ and there exists $U'\subset B_1$ such that $\jl^j(U') \subset B_{j+1}\subset \mathcal{U}_n$ for $0 \le j<k$ and $\jl^{k}(U')=U$. This concludes the proof of statement \textit{(ii)}.
\end{proof}
We can now proceed with proof of Theorem A.

\begin{proof}[Proof of Theorem A]

By \thref{chichicea}, \thref{nusurrconstant}, and \thref{bambampocpoc}, if the connectivity of a Fatou component is different from 1 or 2, then it has to be of the form $\kappa=(n+1)^{i}d^j n^\ell+2$ where $i,j,\ell\in \mathbb{N}$. It follows from the Riemann-Hurwitz formula that if $f:U\rightarrow V$ is proper of degree $q$ without critical points and $V$ has connectivity $p+2$, then $U$ has connectivity $qp+2$. Moreover, these connectivities are achieved through preimages of $\uc$. In order to increase the connectivity we have to take preimages of $\uc$, which has connectivity $3=1+2$. It follows from \thref{haiininterior} that if $U\subset \mathcal{U}_s$ , $s\in \{n+1, \, d, \, n \}$, is a Fatou component that surrounds $z=0$, then it is mapped onto its image with degree $s$. Therefore, in order to increase the coeficient $n$ in the expression of the connectivity, we have to take preimages in $\mathcal{U}_n$. By \thref{haiininterior}, every Fatou component $U \subset \mathcal{U}_n$ is eventually mapped to $\mathcal{U}_d$, without passing through $\mathcal{U}_{n+1}$. By \thref{maxitaxi} \textit{(i)}, for every backwards iteration through $\mathcal{U}_d$ there are at most $k$ backwards iterations in $\mathcal{U}_n$. Since $\uc \not \subset  \mathcal{U}_{d} \cup \mathcal{U}_{n}$, it follows that $\ell \le j k $.
\end{proof}

The final part of this section is devoted to the proof of Theorem B. The following lemma shows that there are no restrictions to the exponents of $n+1$ and $d$ in respect to achievable connectivities.

\begin{lemma}
\thlabel{chichicea1}
Let $\jl$ satisfying (a), (b), (c), and (d). Let $\lambda\neq 0$, $\modulu<\sefu$. Assume also that $\cl \in \uc$, where $\uc$ is an iterated preimage of $\ala$ which surrounds $z=0$. Then, the following hold:

\begin{enumerate}[label=(\roman*)]

\item{There exists an eventual preimage of $\uc$ which lies in $\mathcal{U}_{n+1}$, surrounds $z=0$, and has connectivity $\kappa=(n+1)^i+2, \quad \forall i \in \mathbb{N}$.}
\item{Let $V$ be an eventual preimage of $\uc$ which surrounds $z=0$ and let $\kappa$ be the connectivity of $V$. Then, there exists a Fatou component, which surrounds $z=0$, of connectivity $d(\kappa-2)+2.$}

\end{enumerate}
In particular, there exists a Fatou component which surrounds $z=0$ and has connectivity $\kappa= (n+1)^i d^j+2, \quad \forall i,j \in \mathbb{N}$.
\end{lemma}

\begin{proof}

First we prove \textit{(i)}. Recall that, by \thref{dumnezeu}, $\uc \subset \mathcal{V}_{n+1}$. By \thref{chichicea} and \thref{haiininterior}, for any $i \ge 1$ there exists a Fatou component $U$ which surrounds $z=0$ such that $\jl^{j} (U) \subset \mathcal{U}_{n+1}$, for $j=0, \dots i-1$, and $\jl ^{i}(U)= \uc$. Since $\uc$ has connectivity $3$ and no eventual preimage of $\uc$ contains a critical point, it follows by succesively applying the Riemann-Hurwitz formula that the connectivity of $U$ is $(n+1)^i+2$. This proves \textit{(i)}.

To prove \textit{(ii)}, let $V$ be an eventual preimage of $\uc$ which surrounds $z=0$, of connectivity $\kappa$. By \thref{chichicea}, $V$ has a preimage in $\mathcal{U}_d$ which surrounds $z=0$ and which is mapped onto it with degree $d$. Since $V$ cannot contain any critical value, by the Riemann-Hurwitz formula the connectivity of this preimage of $V$ is $d(\kappa-2)+2$. This concludes the proof of \textit{(ii)}.
\end{proof}

We can now proceed with proof of Theorem B.

\begin{proof}[Proof of Theorem B]

Fix $i\geq0,j\geq 0,\ell\geq 0$ such that $\ell \le j(k-1)$. We want to show that there exists a Fatou component $U$ (which will be an iterated preimage of $\uc$) of connectivity $\kappa=(n+1)^{i}d^j n^\ell +2$. 

If $k=1$, then by \thref{chichicea1} the conclusion holds. Otherwise, $\jl (\uc) \subset \mathcal{U}_n$  and $\mathcal{U}_d \subset \mathcal{V}_n$. By \thref{chichicea1} \textit{(i)}, there exists a Fatou component $V$ which is an iterated preimage of $\uc$, surrounds $z=0$, lies in $\mathcal{U}_{n+1}$, and has connectivity $ (n+1)^i+2$. This concludes the proof for $j=0$. 

Assume that $j \ne 0$ (remember that $k>1$). By \thref{chichicea} \textit{(i)}, there exists a preimage of $V$ in $\mathcal{U}_d$ which surrounds $z=0$. By \thref{maxitaxi} \textit{(ii)}, there exists $V^{(1)}\subset \mathcal{U}_n$ which surrounds $z=0$ such that 
$$
\jl^r(V^{(1)}) \subset \mathcal{U}_n \ \ {\textrm for} \ \ 0\leq r < k-1, \quad \jl^{k-1}(V^{(1)})\subset \mathcal{U}_d \quad {\textrm and} \quad \jl^{k}(V^{(1)})=V.
$$
Recall that no iterated preimage of $\uc$ contains a critical point and so, by \thref{chichicea}, if they lie in $\mathcal{U}_i$, $ i \in \{n+1, \, n, \, d \} $, they map forward with degree $i$. Applying this criteria to the iterated preimages of $V$ up to $V^{(1)}$, we get from the Riemann-Hurwitz formula that 
$$
\kappa\left(\jl^r(V^{(1)})\right)=(n+1)^{i}dn^{k-1-r}+2,\ r=0,\ldots k-1.
$$
Starting the process with $V^{(1)}$ instead of $V$ we can take a preimage of $V^{(1)}$ in $\mathcal U_d$ and then we can take up to $k-1$ iterated preimages in $\mathcal U_n$ to land on, say, $V^{(2)}$. As above, we get that  
$$
\kappa\left(\jl^r(V^{(2)})\right)=(n+1)^{i}d^2n^{2(k-1)-r}+2,\ r=0,\ldots k-1.
$$
Repeating the same process $j$-times we conclude that there exist Fatou components with connectivity
$$
\kappa\left(\jl^r(V^{(s)})\right)=(n+1)^{i}d^sn^{s(k-1)-r}+2,\ s=1,\ldots j,\  r=0,\ldots k-1.
$$
If $(j-1)(k-1)<\ell\leq j(k-1)$ we are done (take $s=j$ in the previous formula). If $\ell \leq t(k-1)$ with $t\leq (j-1)$, then we stop the process at level $t$ and take $j-t$ preimages in $U_d$ to get the desired connectivity.

\end{proof}

\section{Proof of theorem C}\label{sec:Theorem_C}
In this section we prove Theorem C.  We first show that there is a sequence of preimages of  $A_{\la}$ which surround $z=0$ and accumulate on $\partial \bl$. We want  to remark that the set $A^{\textrm {out}}$  depends on $\lambda$ even if we do not indicate it in its notation.
\begin{lemma}
\thlabel{alambdas}
Let $\jl$ satisfying (a), (b), (c), and (d). Let $\lambda\neq 0$, $\modulu<\sefu$. Let $A_{0, \la}:=\ala$. Then there exist $\{A_{m, \la}\}_{m \ge 1}$ iterated preimages of $A_{0, \la}$, $\jl(A_{m+1, \la})=A_{m, \la}$, such that the following properties are satisfied:
\begin{enumerate}[label=(\roman*)]
\item{Each Fatou component $A_{m, \la}$ is surrounded by $A_{m+1,\la}$, that is, $A_{m, \la}\subset \textrm{Bdd} (A_{m+1,\la})$. In particular, all $\{A_{m, \la}\}_{m \ge 1}$ lie in $A^{\textrm out}$.} 
\item{The sequence of Fatou components $\{A_{m, \la}\}_{m \ge 1}$ accumulate on $\partial \bl$ as $m\to \infty$.}
\end{enumerate}
\end{lemma}
\begin{proof}
Every Fatou component which surrounds $z=0$ has exactly 2 boundary components which surround $z=0$. It follows from \thref{prop:preimcurves} that every Fatou component in $A^{\textrm{out}}$ which surrounds $z=0$ has exactly a preimage in $A^{\textrm {out}}$ which surrounds $z=0$. Let $\{A_{m, \la}\}_{m \ge 1}$ be the sequence of Fatou components obtained by taking consecutive preimages of $A_{0, \la}$ in $A^{\textrm{out}}$ which surround $z=0$. Since $\jl$ is conjugated to a Blaschke product on $A^{\textrm {out}}$, by \thref{prop:conjblas}, the Fatou components $A_{m, \la}$ accumulate on $\partial \bl$ as $m$ goes to $\infty$. It also follows from the conjugation with the Blaschke product that  $A_{m, \la}\subset \textrm{Bdd} (A_{m+1,\la})$ for all $m\geq 0$.
\end{proof}

The multiply connected sets  $A_{m, \la}$ surround $z=0$. Therefore, there are exactly 2 boundary components of $A_{m, \la}$ which surround $z=0$. We denote them by $\partial^{\textrm{Int}} A_{m, \la}$ and $\partial^{\textrm{Ext}} A_{m, \la}$, where $\partial^{\textrm{Int}} A_{m, \la} \subset \textrm{Int}(\partial^{\textrm{Ext}} A_{m, \la})$. Next lemma tells as that if $m$ is large enough then there are parameters $\lambda$ such that $\nu_{\la}\in A_{m, \la}$.

\begin{lemma}
Let $\jl$ satisfying (a), (b), (c), and (d). Let $\lambda\neq 0$, $\modulu<\sefu$. Then, if $m \in \mathbb{N}^{*}$ is big enough, there exists a parameter $\la$ such that $\cl \in A_{m, \la}$.
\thlabel{lem5point2}
\end{lemma}

\begin{proof}
	
The idea of the proof is to show that, if $m$ is big enough, we can find a parameter $\la_{0}$ such that $\nu_{\la_{0}}\in\textrm{Bdd}(A_{m, \la_{0}})$ and a parameter $\la_{m}$ such that $\nu_{\la_{m}}$ belongs to the unbounded component of $\mathbb{C}^*\setminus A_{m, \la_m}$. We will then conclude that there exists a parameter $\la'_{m}$ such that $\nu_{\la'_{m}}\in A_{m, \la'_{m}}$.

Fix $\lambda_0$ such that all hypothesis hold. Then $A_{0, \la_0}$ is well defined, and so are $A_{m,\la_0}$, $m> 0$. Let $m_0$ be such that $\nu_{\la_{0}}\in\textrm{Bdd}(A_{m_0, \la_{0}})$. Then, for all $m\geq m_0$ we have $\nu_{\la_{0}}\in\textrm{Bdd}(A_{m, \la_{0}})$. 
For fixed $m\geq m_0$, we want to find the parameter $\la_{m}$. If $\la=0$, the critical point $\nu_0$ belongs to the boundary of the maximum domain of definition of the Böttcher coordinate of $A^*(0)$ under $S_{n,a,Q}$. Therefore, the orbit of $\nu_0$   under $S_{n,a,Q}$ accumulates on $z=0$ but never maps onto it. Observe that $\jl $ converges uniformly to $\jn$ on compact subsets of $\mathbb{C}^*\setminus \mathbb{D}_{\epsilon}$ as $\lambda\rightarrow 0$, where $\epsilon>0$ is arbitrarily small and $\mathbb{D}_{\epsilon}$ denotes the disk of radius $\epsilon$ centered at $z=0$. Consequently, for fixed $m\geq 0$ and $\epsilon>0$, if $|\lambda|$ is small enough then $A_{m, \la}\subset \mathbb{D}_{\epsilon}$. Since $\nu_{\la}\rightarrow\nu_0$ as $\la\rightarrow 0$, it follows that if $|\lambda|$ is small enough then $\nu_{\la}$ belongs to the unbounded component of $\mathbb{C}^*\setminus A_{m, \la}$. It is enough to take $\la_m$ to be any such $\la$.

\begin{figure}
	\subfigure{ \includegraphics[width=5cm]{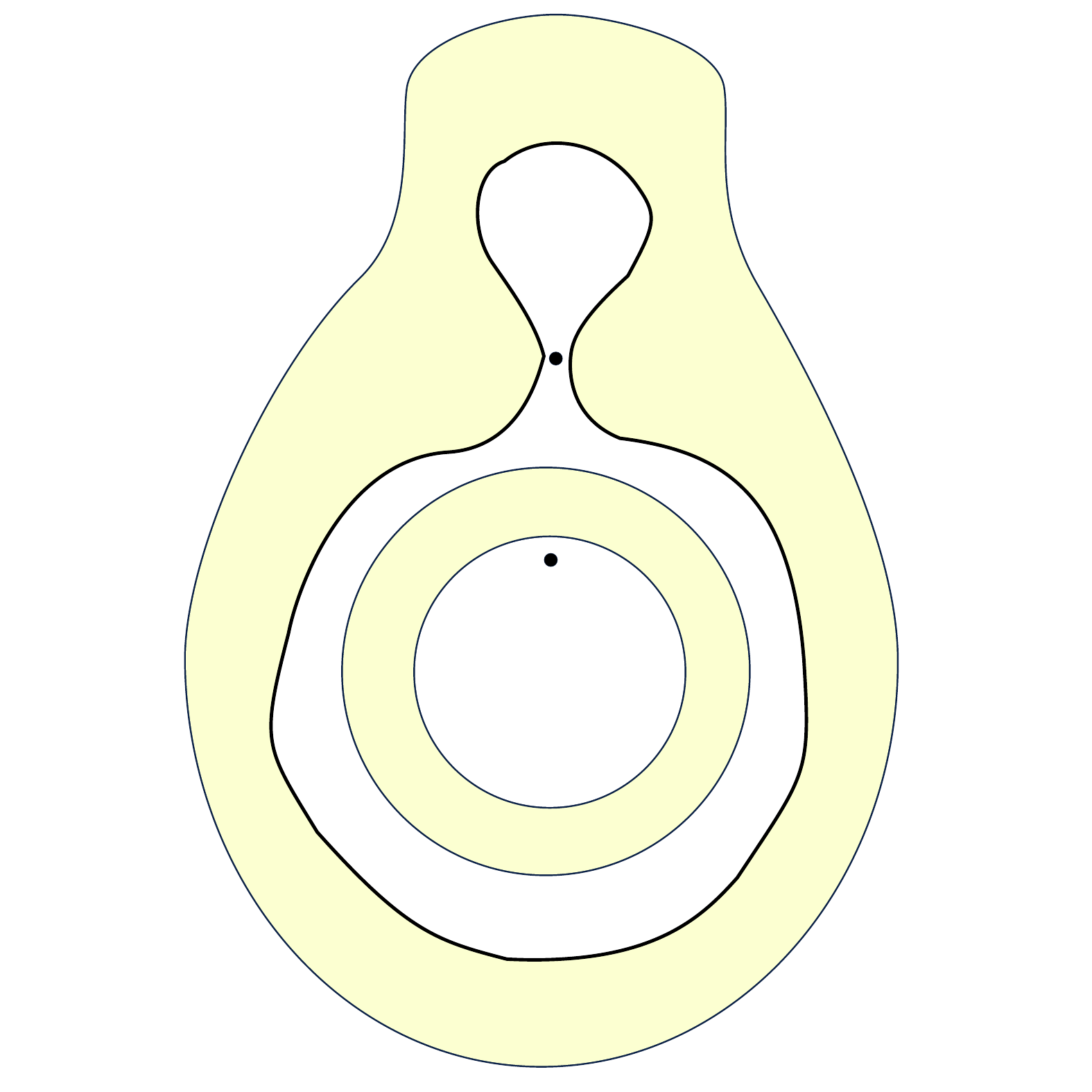}
             \put(-85, 95){$\cl$}
              \put(-82,65){$\jl(\cl)$}
 \put(-85,30){$A_{k-1, \la}$}
       \put(-85,7){$A_{k, \la}$}
	\includegraphics[width=5cm]{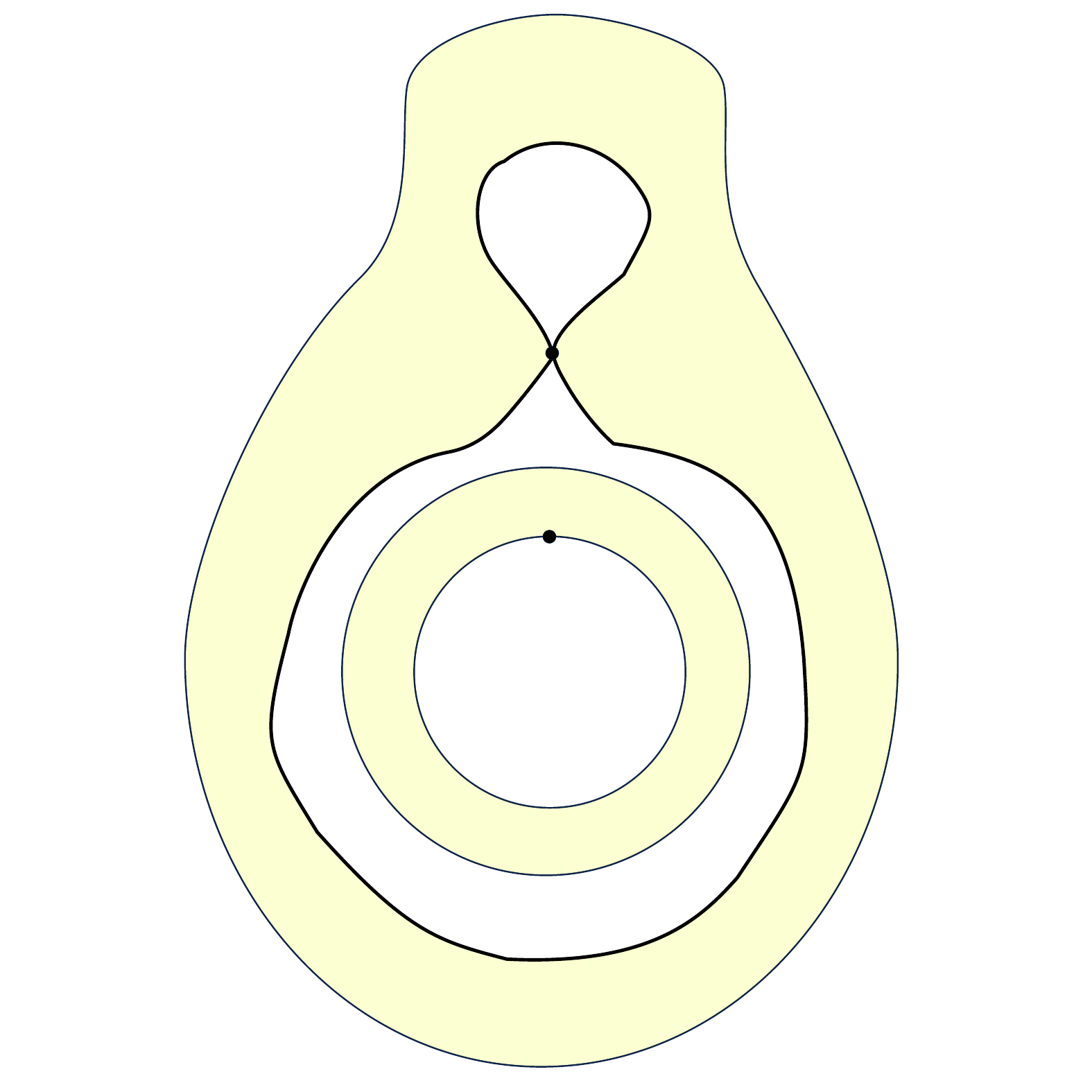}
    \put(-85,95){$\cl$}
       \put(-82,65){$\jl(\cl)$}
    \put(-85,30){$A_{k-1, \la}$}
       \put(-85,7){$A_{k, \la}$}
	\includegraphics[width=5cm]{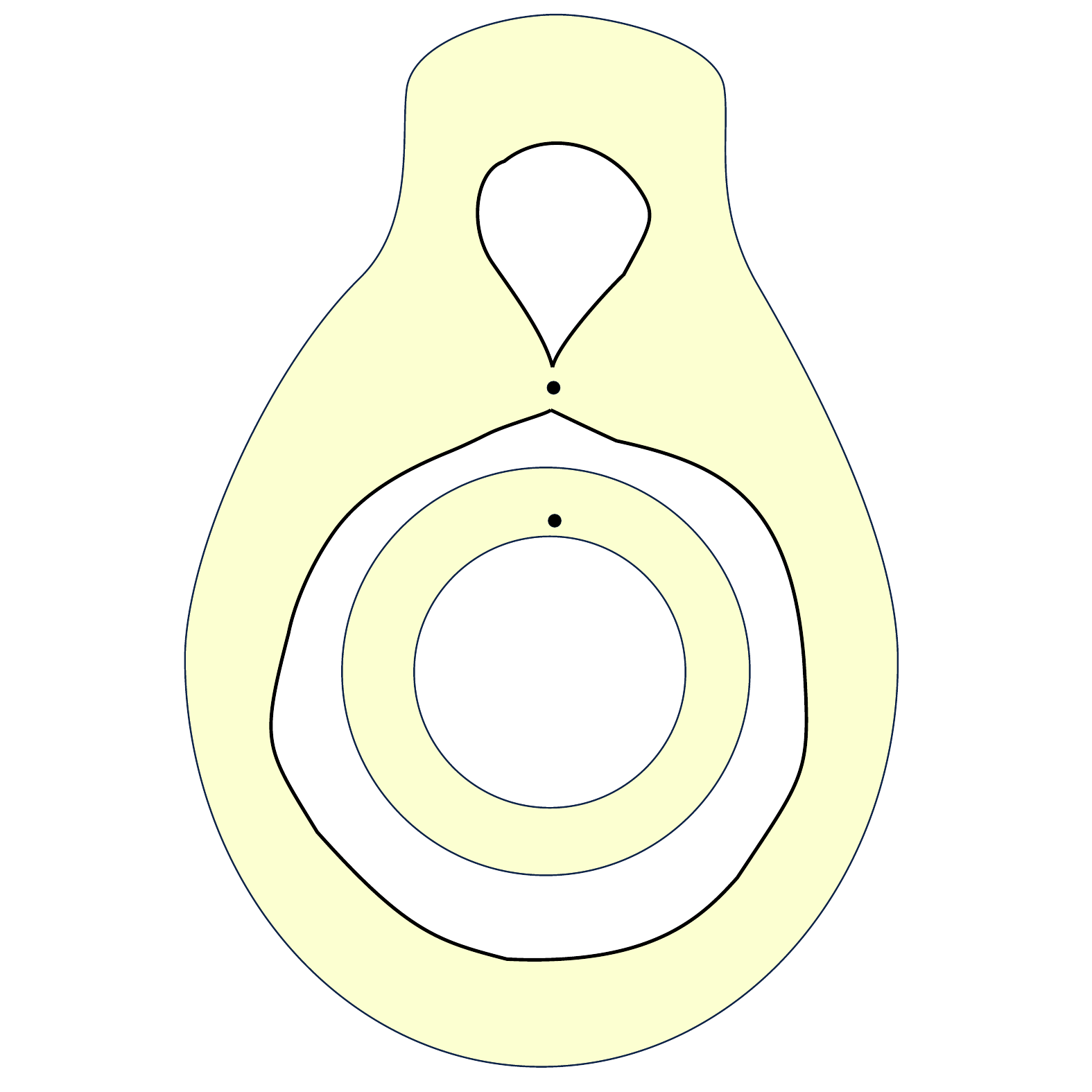}
    \put(-83,92){$\cl$}
                   \put(-82,65){$\jl(\cl)$}
 \put(-85,30){$A_{k-1, \la}$}
             \put(-85,7){$A_{k, \la}$}}
	\subfigure{\includegraphics[width=5cm]{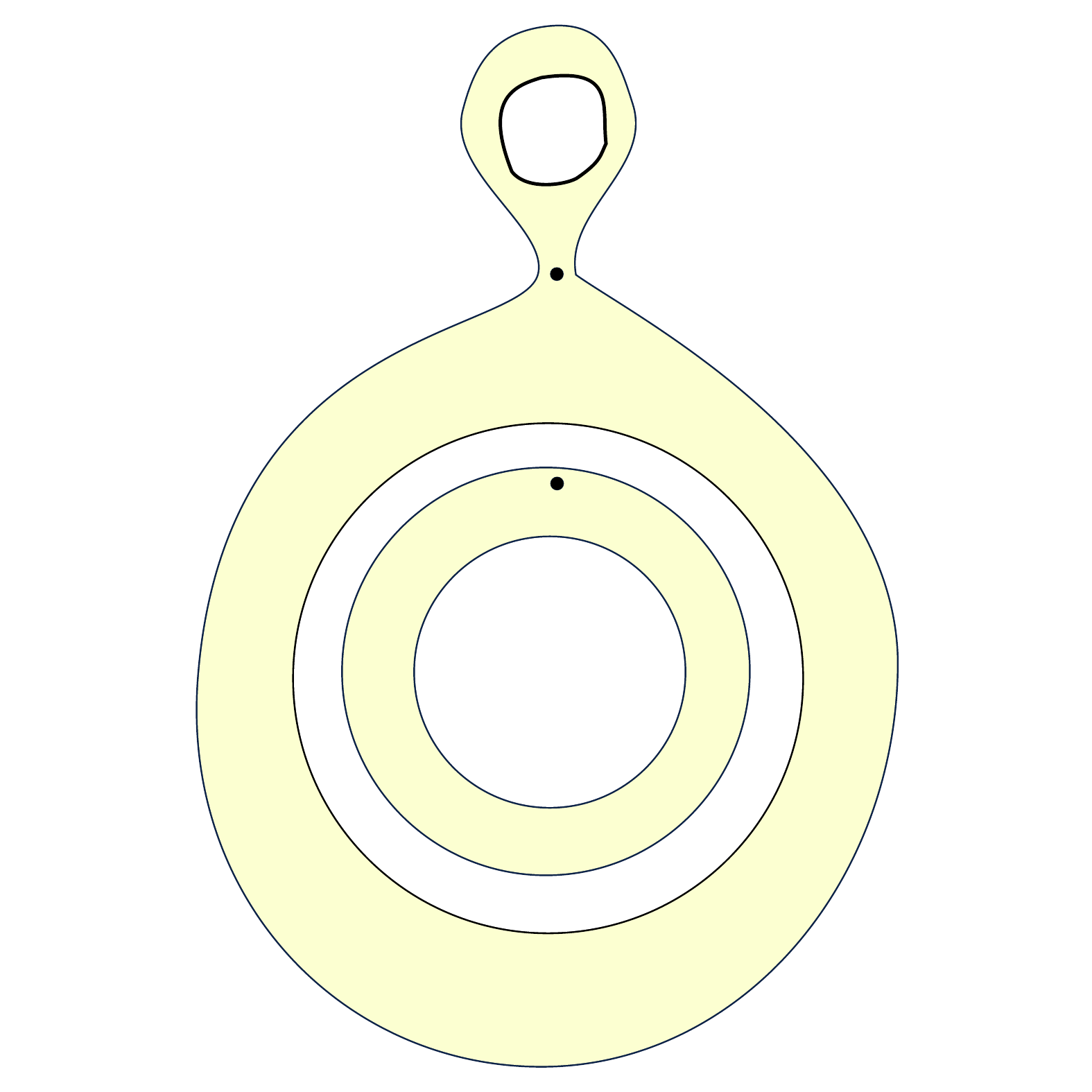}
 \put(-85,107){$\cl$}
              \put(-85,73){$\jl(\cl)$}
 \put(-85,30){$A_{k-1, \la}$}
       \put(-85,7){$A_{k, \la}$}
	\includegraphics[width=5cm]{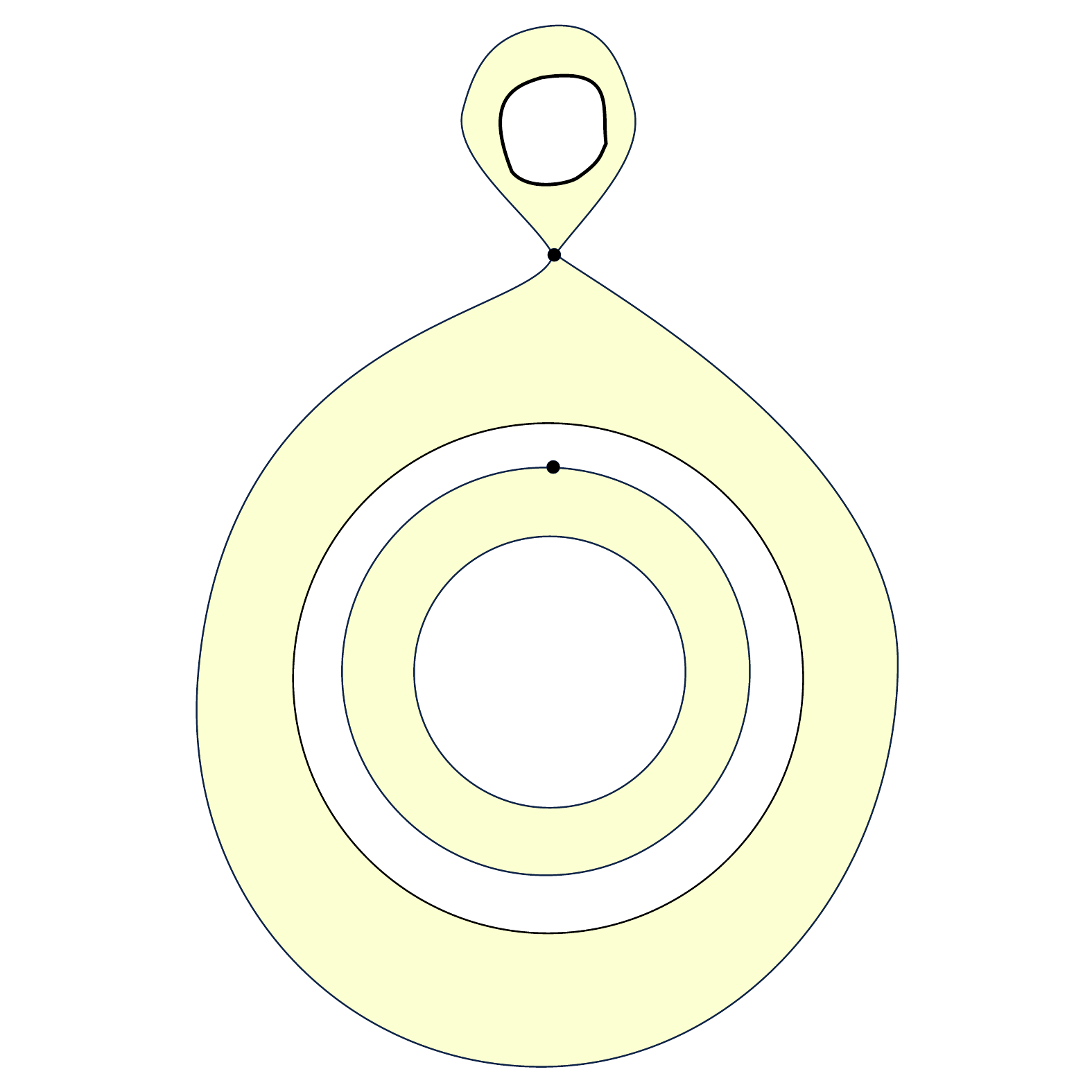}
 \put(-84,108){$\cl$}
              \put(-85,76){$\jl(\cl)$}
 \put(-85,30){$A_{k-1, \la}$}
       \put(-85,7){$A_{k, \la}$}
          \put(-75,133){\tiny{$A'$}}
	\includegraphics[width=5cm]{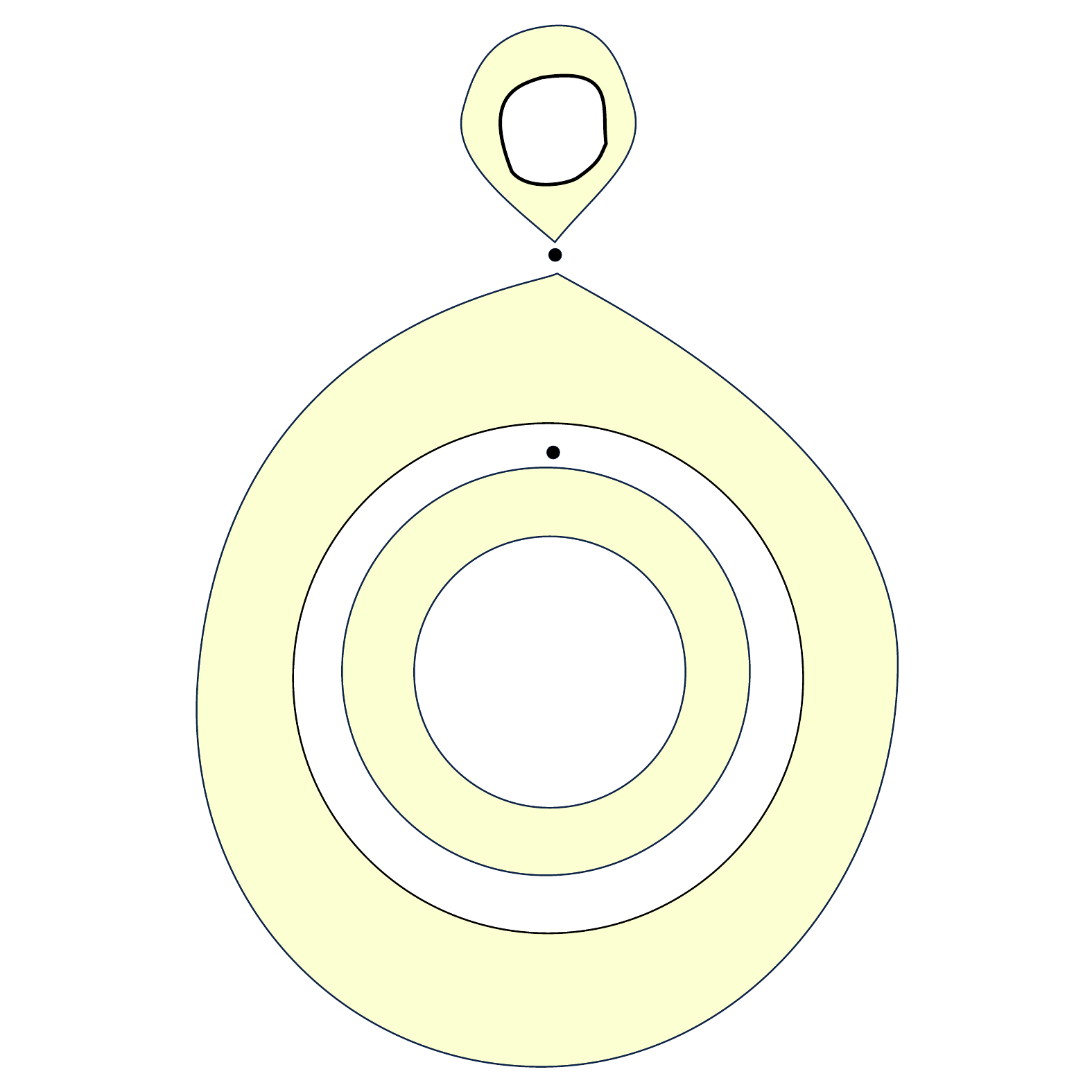}
 \put(-85,108){$\cl$}
              \put(-85,76){$\jl(\cl)$}
 \put(-85,30){$A_{k-1, \la}$}
       \put(-85,7){$A_{k, \la}$}}
 \caption{The top figures correspond to the possible cases of $\nu_{\la}$ lying in a neighbourhood of $\partial^\textrm{Int} A_{m,\la}$. The top figures correspond to the possible cases of $\nu_{\la}$ lying in a neighbourhood of $\partial^\textrm{Ext} A_{m,\la}$.}
	\label{fig:nearbif}
\end{figure}
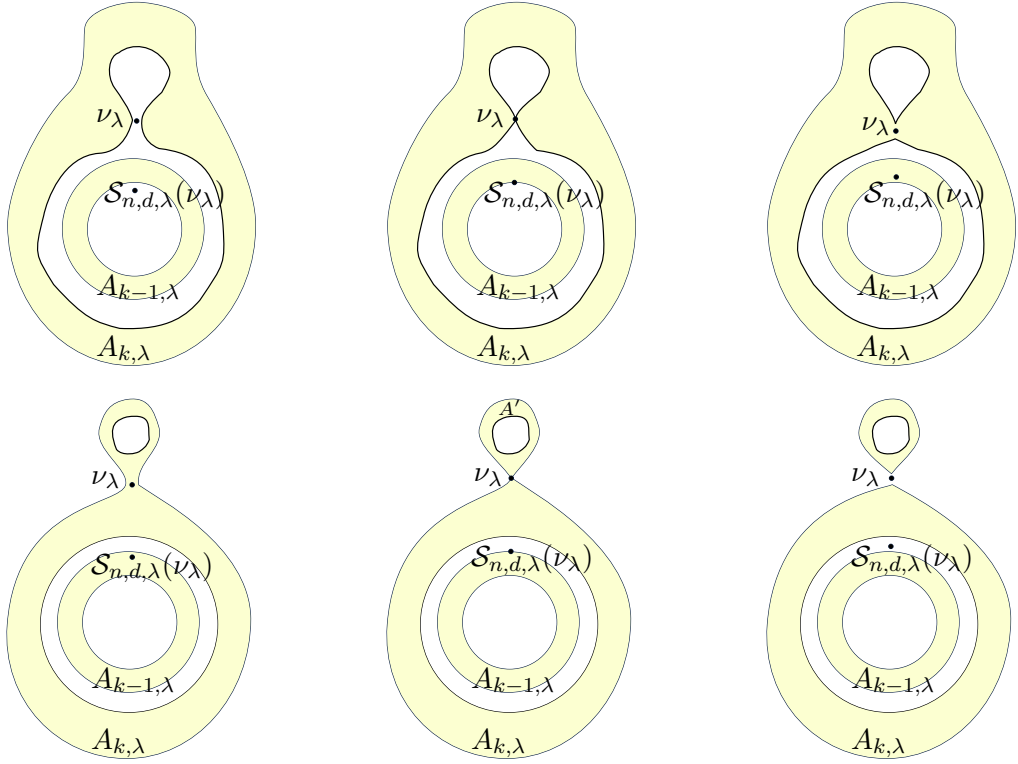

To finish the proof we need to show that when we move continuously the parameter from $\la_0$ to $\la_m$ we need to find intermediate parameters $\la'_m$ such that $\nu_{\la'_{m}}\in A_{m, \la'_{m}}$.  By \thref{prop:continfty} we know that $\partial \mathcal{A}^*_{\infty}$ moves continuously with respect to $\la$. Since $\partial \mathcal{A}^*_{\infty}$ and $\partial T_{\lambda}$ cannot contain critical values, it follows that both boundary components of $A_{0,\la}$ move continuously with respect to $\la$. 
For fixed $\la'$, the set $\partial{A_{m,\la}}$, $m\geq1$,  moves continuously with respect to $\lambda$ in a neighbourhood of $\la'$ unless $\partial A_{m,\la'}$ (or an iterated image of $\partial A_{m,\la'}$) contains a critical point. Here by moving continuously we mean that every connected component of $\partial{A_{m,\la}}$ is a Jordan curve that moves continuously with respect to the Hausdorff metric (in particular, it does not pinch itself or split in several connected components). Notice that since there is only a free critical point, at most one of the 2 components of $\partial A_{m,\la}$ which surround $z=0$ may not move continuously for $\la$ in a neighbourhood of $\la'$. Using \thref{prop:conjblas} it can be proven that $\jl(\partial^{\textrm{Ext}} A_{m, \la_0})=\partial^{\textrm{Ext}} A_{m-1, \la_0}$ and $\jl(\partial^{\textrm{Int}} A_{m, \la_0})=\partial^{\textrm{Int}} A_{m-1, \la_0}$. Assume that for $\la'$ we have  $\nu_{\la'}\in\partial A_{m,\la'}$. Then $A_{m-1,\la}$ is an annulus that moves continuously with respect to $\la$ for all $\la$ in a neighbourhood of $\la'$. If $\nu_{\la'} \in \partial^\textrm{Int} A_{m,\la'}$ then $\jl(\nu_{\la'}) \in \partial^\textrm{Int} A_{m-1,\la'}$. By \thref{prop:preimcurves}, for $\la$ in a neighbourhood of $\la'$ exactly one of the following holds (see the three upper figures in Figure~\ref{fig:nearbif}):
 \begin{itemize}
 	\item If  $\jl(\nu_{\la}) \in \textrm{Int}(\partial^\textrm{Int} A_{m-1,\la})$ then  $\nu_{\la} \in \textrm{Int}(\partial^\textrm{Int} A_{m,\la})$ and $A_{m,\la}$ is doubly connected.
 	\item If  $\jl(\nu_{\la}) \in \partial^\textrm{Int} A_{m-1,\la}$ then  $\nu_{\la} \in \partial^\textrm{Int} A_{m,\la}$. Then,  $A_{m,\la}$ is doubly connected and $\partial^\textrm{Int} A_{m,\la}$ consists of the union of 2 Jordan curves.
 	\item If  $\jl(\nu_{\la}) \in  A_{m-1,\la}$ then  $\nu_{\la} \in  A_{m,\la}$ and $A_{m,\la}$ is triply connected.
 \end{itemize}

On the other hand,  if $\nu_{\la'} \in \partial^\textrm{Ext} A_{m,\la'}$ then $\jl(\nu_{\la'}) \in \partial^\textrm{Ext} A_{m-1,\la'}$. By \thref{prop:preimcurves}, for $\la$ in a neighbourhood of $\la'$ exactly one of the following holds (see the three lower figures in Figure~\ref{fig:nearbif}):
\begin{itemize}
	\item If  $\jl(\nu_{\la}) \in  A_{m-1,\la}$ then  $\nu_{\la} \in  A_{m,\la}$ and $A_{m,\la}$ is triply connected.
	\item If  $\jl(\nu_{\la}) \in \partial^\textrm{Ext} A_{m-1,\la}$ then  $\nu_{\la} \in \partial^\textrm{Ext} A_{m,\la}$. Then,  $A_{m,\la}$ is doubly connected, $\partial^\textrm{Ext} A_{m,\la}$ is a Jordan curve, and there is an extra preimage $A'$ of $A_{m-1,\la}$ such that $\partial A_{m,\la}\cap \partial A'=\nu_{\la}$.
	\item If  $\jl(\nu_{\la}) \in \textrm{Ext}(\partial^\textrm{Int} A_{m-1,\la})$ then  $\nu_{\la} \in \textrm{Ext}(\partial^\textrm{Int} A_{m,\la})$ and $A_{m,\la}$ is doubly connected.
\end{itemize}

It follows from the previous configurations that if we move continuously  the parameter $\la$ from $\la_0$ until $\la_m$ we can find parameters $\la'_m$ such that $\nu_{\la'_{m}}\in A_{m, \la'_{m}}$. This finishes the proof of the result.

\end{proof}
We can now proceed with proof of Theorem C.

\begin{proof}[Proof of Theorem C]

Fix $i, \, j, \, \ell$. We have to prove that there exists $\la$ for which there is a Fatou component of connectivity $\kappa=(n+1)^{i}d^j n^\ell+2$ and a Fatou component of connectivity $\kappa=(n+1)^i+2$. Recall that the results in Section 4 required the free critical point $\cl$ to lie in a preimage of $\ala$ which surrounds $z=0$. By \thref{lem5point2}, there exists $m > \ell$ such that $\cl \in A_{m, \la}$. The existence of the Fatou component of connectivity $\kappa=(n+1)^i+2$ is proven in \thref{chichicea1}\textit{(i)}. Since $\cl \in A_{m, \la}$ and $m> \ell$, the existence of a Fatou component of connectivity $\kappa=(n+1)^{i}d^j n^\ell+2$ follows from Theorem B. 

\end{proof}

\bibliography{bibliografia}
\bibliographystyle{amsalpha}
\end{document}